%% file: reversal_paper.tex
\documentclass[article]{siamart}
\usepackage{amsmath,amssymb}
\usepackage{caption}
\usepackage{hyperref}
\usepackage[utf8]{inputenc}
\usepackage{mathtools}
\usepackage{color,graphicx}
\usepackage{thmtools, thm-restate}
\usepackage{enumitem}
\usepackage{algorithmicx}
\usepackage{algpseudocode}
\usepackage{tikz}
\usetikzlibrary{arrows}
\usetikzlibrary{3d,calc}

\input{macros.tex}

\begin{document}
  
\title{{\TheTitle}}

\author{Donsub Rim%
  \thanks{Department of Applied Mathematics, University of Washington, Seattle WA 98195 (\email{drim@uw.edu}, \email{smoe@uw.edu}, \email{rjl@uw.edu}).}%
  \and
  Scott Moe\footnotemark[1]
  \and
  Randall J. LeVeque\footnotemark[1]
}

\maketitle

\begin{abstract} 
   Snapshot matrices built from solutions to hyperbolic partial differential 
   equations exhibit slow decay in singular values, whereas fast decay is 
   crucial for the success of projection-based model reduction methods.
   To overcome this problem, we build on previous work in symmetry 
   reduction [Rowley and Marsden, Physica D (2000), pp. 1-19] 
   and propose an iterative algorithm 
   that decomposes the snapshot matrix into multiple shifting profiles, 
   each with a corresponding speed. 
   Its applicability to typical hyperbolic problems is demonstrated through 
   numerical examples, and other natural extensions that modify the shift 
   operator are considered. Finally, we give a geometric interpretation
   of the algorithm.
\end{abstract}

\section{Introduction}

Reduced order models (ROMs) can emulate the behavior of high-dimensional 
models (HDMs) with small computational cost.  
Therefore, once ROMs can be constructed, prohibitively expensive problems 
in control design or uncertainty quantification (UQ) can be tackled by 
using them as surrogate models in place of HDMs.
Proper orthogonal decomposition (POD) and its variants 
\cite{pod02,pod03,pod06,pod07,gappy-pod,pod17,pod18,pod23,pod27}  
have been successfully applied to various partial differential equations (PDEs),
including those arising in fluid dynamics 
\cite{pod01,pod04,pod08,pod15,pod20,pod26}.
However, these projection-based methods can be ineffective 
when applied to compressible flow problems governed by hyperbolic PDEs. 
This difficulty is well-known and was noted 
in \cite{amsallem} where a dictionary-based model reduction 
method was developed, and in \cite{Carlberg15} where a fail-safe $h$-adaptive
algorithm was introduced.

We will illustrate the main obstacle with a simple example.
Consider the initial boundary value problem for the advection equation, 
whose solution $u$ in the domain $\Dom \equiv (0,1)$ satisfies the PDE
\begin{equation}
u_{t} + c {u}_{x}  = 0 \quad \quad \Iin \Dom, 
\label{eq-advection}
\end{equation}
along with the periodic boundary condition and the initial condition
\begin{equation}
u(0,t) = u(1,t) \Ffor t \in [0,T], \quad \quad
u(x,0) = u_0(x) \equiv \delta(x).
\end{equation}
We assume $c = 1$ here. 
Let us seek a solution using the finite volume method (FVM) with 
upwind flux \cite{fvmbook}. We set the grid points
$x_j = jh$ for $j=0,1,...,N$ and let $h \equiv 1/N$,
and define the cells $\cC_j \equiv [x_{j-1/2}, x_{j+1/2}]$
where $x_{j+1/2} \equiv x_j + h/2$.
Denote by $u_j^n$ the approximation to the cell average of 
the solution at time $t_n$
\beq
u_j^n \approx \frac{1}{h}\int_{x_{j-1/2}}^{x_{j+1/2}} u(x,t_n) \dx,
\label{eq-cellavg}
\eeq
and also denote by $\bu^n$ the vector $(u_j^n)_{j=1}^N \in \RR^{N}$.
Taking a fixed time-step of size $\Delta t = h$, 
it is easy to see that the finite volume solution at time $t_n$ is just the 
scaled standard basis vector $\bu^n = \be_{n} / h$. That is,
\beq
    \bu^n = \bfK^n \bu^0,
    \where \quad
    \bfK \equiv \begin{bmatrix} 0 & 0 &  \cdots & 0 & 1 \\
                           1 & 0 &  \cdots & 0 & 0 \\
                            0 & 1 &  \cdots & 0 & 0 \\
    \vdots & \vdots& \ddots & \vdots &\vdots \\
    0 & 0 & \cdots &1 & 0 
            \end{bmatrix} \Aand \bu^0 = \frac{1}{h} \be_1.
\label{eq-ku0}
\eeq
This solution has no error apart from the discretization of the initial data, 
and thereafter reproduces the cell averages \eqref{eq-cellavg} exactly.

The \emph{snapshot matrix} $\bfA$, taken at times $\{t_n\}_{n=0}^{N-1}$, 
is given by
\beq
\bfA \equiv \begin{bmatrix} \bu^0 & \bu^1 & \cdots & \bu^{N-1} \end{bmatrix} 
= \frac{1}{h}\bfI, \where \bfI \text{ is the identity in } \RR^{N\times N}. 
    \label{eq-snapshot} 
\eeq
In POD, we take the 
singular value decomposition (SVD) of the matrix $\bfA$.
It is easy to see that $\bfA$ has singular values 
$\sigma_1 = \cdots = \sigma_N = 1/h$.
We then truncate the rank-$1$ expansion of $\bfA$ after some $R$ terms.
Usually, we choose the smallest $R$ such that, for a given 
tolerance $\eps \ll 1$, the remainder satisfies
\beq
\left. \sum^{N}_{j = R+1} \sigma_j^2 \right/ \sum_{j=1}^{N} \sigma_j^2 < \eps,
    \quad
    \text{ in this case }
    1 - R/N < \eps.
	\label{eq-thresh}
\eeq
The LHS decreases linearly in $R$, so $R$ must be large
even for a moderately small $\eps$ to satisfy \eqref{eq-thresh}. 
If $\eps < h$ it would require
$R=N$, so that all singular values and corresponding singular vectors 
must be kept as reduced basis vectors.
(Often $\bfA$ is preprocessed by subtracting from each column its mean,
so that each column has zero mean. Doing so here would make $\sigma_N$ equal to 
zero, but other singular values will not be changed, leaving the obstacle
intact.)

This slow decay is commonly observed in snapshot matrices taken from
hyperbolic problems. Therefore, existing projection-based methods 
quickly face a difficulty. The approach we adopt to overcome this 
is to focus on the low-dimensional \emph{hyperbolic} behavior of the 
solution and treat it separately, to the extent possible. 
Simply put, we wish to construct the Lagrangian frame of reference.  
To do so directly is a challenging problem of its own right, so instead 
we devise a numerical method for utilizing this frame indirectly
for our special purposes.

This main idea coincides with the so-called \emph{symmetry reduction} 
that was studied in \cite{marsden2, marsden1} and similar ideas
that appeared in the references therein.
The target for reduction in that context is a continuous symmetry group $G$ 
acting on a manifold $M$. In our setting, $M$ is the $L^2$ inner product space 
of periodic functions, and $G$ is the group of spatial translations.
To reduce $G$, \emph{template fitting} \cite{kirby} is used to map the full 
dynamics of $u$ to the quotient space $M/G$.
Given a snapshot $u(x,t)$ at time $t$ and a \emph{template} $u_0(x)$,
both periodic in $[0,1]$, template fitting posits the minimization problem
\beq
\min_a \int_0^{1} \abs{u(x-a,t) - u_0(x)}^2 \dx.
\label{eq-contimin}
\eeq
This minimization resembles the orthogonal Procrustes problem 
\cite{gower2004procrustes}, which deals with data given in the form of 
sample points, rather than in discretized function values over a grid.
If $u$ is smooth, one obtains the equation for the minimum $a_*$,
$\brkt{u(x,t)}{u_0'(x+a_*)}= 0$, that defines the dynamics of $a_*(t)$.
The orthogonality condition then allows one to identify the quotient space 
$M/G$ with an affine space intersecting $u_0$ called a \emph{slice} denoted 
by $S_{u_0}$. To summarize, for each given dynamics $u(t)$ in $M$,
corresponding \emph{slice dynamics} $r(\tau)$ in $S_{u_0}$ can be found. 
After a reduction for $r(\tau)$ is found in $S_{u_0}$, there are 
reconstruction equations that can be used to recover 
the original dynamics $u(t)$ \cite{marsden1}. 
The main advantage is that $r(\tau)$ in the space $S_{u_0}$ may yield
low-dimensional structure more readily, even when $u(t)$ itself does not.
This key property is the inspiration for this work.

Following the template fitting approach, we propose generalizations which 
expand its applicability.
In \cite{marsden1} the dynamics of the infinitesimal action for the 
reconstruction were formulated, and then the system was integrated numerically. 
Here we consider the direct discretization of \eqref{eq-contimin}, then
devise a greedy algorithm we call the \emph{transport reversal}.
The algorithm terminates when the snapshot matrix can be well-approximated
by the superposition of multiple transport dynamics.
The main ideas are 
(1) the projection onto the \emph{template} or the \emph{pivot} for scaling, 
(2) the use of \emph{cut-off vectors} to modify the pivot, and 
(3) enforcing of regularity in the minimization problem to obtain 
    smooth transport dynamics. 
The details appear in Section \ref{sec:rev}.

In the subsequent sections, we consider two extensions of the shift operator. 
The upwind flux is used to extend the shift numbers to real numbers
in Section \ref{sec:upwind}. Then an extension to the case where
the speed $c$ in \eqref{eq-advection} varies with respect to the spatial 
variable is introduced in Section \ref{sec:varlin}.
Nothing prevents these extensions from being used in conjunction with the 
iterative transport reversal algorithm introduced in the preceding section.
In Section \ref{sec:geo}, we present some geometric interpretations.

Transport reversal shares features with the shifted proper 
orthogonal decomposition (sPOD) introduced in \cite{Schulze15}.
It can be related to the dynamic mode decomposition (DMD) \cite{dmd1,dmd2} 
in the sense that the periodic shift operator in \eqref{eq-ku0} 
is a linear operator generating the dynamics on the state space $\RR^N$,
but the objectives differ. Here we assume that a specific dynamic, 
namely transport, is present in the data, whereas DMD aims to discover 
the spectral properties of the Koopman operator derived from the data itself. 

The discovery of the hyperbolic structure through this algorithm
is only a first step towards building a ROM for hyperbolic PDEs.
Using this output to build a ROM requires tackling further issues
that will be pursued in future work.
Once ROMs can be constructed for any parameter value, they can be used 
to explore the solution behavior in parameter space. 
In many practical applications the parameter space is 
high dimensional, so one needs a strategy for constructing a global model 
that is not sensitive to the number of dimensions. 
To this end, various interpolation methods incorporating adaptive and greedy 
strategies have been introduced \cite{param5,param2,param3,param1,param4}.
The algorithms in this paper allows one to apply these methods in 
conjunction with \emph{displacement interpolation} 
(see, e.g., \cite{villani2008optimal}) thereby incorporating the 
Lagrangian frame into the approximation procedure. 
The approach given here may well supplement not only existing
model reduction methods, but also UQ methods such as 
the generalized Polynomial Chaos (gPC) \cite{Pulch2011}.

The transport reversal extends naturally to the multidimensional setting.
The key component in the extension is the use of the intertwining property of
the Radon transform \cite{Helgason2011}. In exploiting this remarkable
property, one obtains a multidimensional extension of the large time-step
method \cite{largetimestep1, largetimestep3, largetimestep2}, and therefore
the multidimensional analogue of the transport reversal algorithm.
The scope of this paper does not permit a detailed account of this 
important extension. A thorough treatment will appear elsewhere, based
on the one-dimensional algorithm presented in this paper.

\section{Transport reversal}\label{sec:rev}

In this section, we discretize and generalize the problem \eqref{eq-contimin} and
then introduce the transport reversal algorithm. 
To motivate the discussion, let us revisit the problem \eqref{eq-advection}.
Recall the finite volume solution $\bu^n$, the matrix of shifts with periodic
boundary conditions $\bfK$ \eqref{eq-ku0}, and the snapshot matrix  
$\bfA$  \eqref{eq-snapshot}.
With this notation, the columns of $\bfA$ can be rewritten in terms of 
the Krylov subspace generated by $\bfK$ in using the fact that 
$\bu^n = \bfK^n \bu^0$,
\[
\bfA = \begin{bmatrix}
        \bu^0 & \bfK \bu^0 & \cdots & \bfK^{N-1} \bu^0
        \end{bmatrix}.
\]
Suppose we preprocess $\bfA$ to obtain $\mathring{\bfA}$,
\beq
    \mathring{\bfA} \equiv 
\begin{bmatrix} \bfI & \bfK^{-1} & \bfK^{-2} 
& \cdots & \bfK^{-(N-1)} \end{bmatrix} \odot 
\begin{bmatrix} \bu^0 & \bu^1 & \bu^2 
& \cdots & \bu^{N-1} \end{bmatrix}\label{eq-const_reverse}
\eeq
where the notation $\odot$ denotes component-wise multiplication between 
a list of matrices and a list of column vectors.
It follows that
\beq
\mathring{\bfA}  = \begin{bmatrix} \bu^0 &  \cdots & \bu^0  \end{bmatrix} 
           = 
           \begin{bmatrix} \frac{1}{h} & \frac{1}{h} &  \cdots & 
           \frac{1}{h} \\
                    0 & 0 & \cdots & 0 \\
                \vdots & \vdots & \ddots & \vdots \\
                0 & 0 & \cdots & 0 
            \end{bmatrix}
            =
           \begin{bmatrix} N & N &  \cdots & 
           N \\
                    0 & 0 & \cdots & 0 \\
                \vdots & \vdots & \ddots & \vdots \\
                0 & 0 & \cdots & 0 
            \end{bmatrix},
\eeq
and also that $\mathring{\bfA}$ has the trivial SVD
\beq
    \mathring{\bfA} = \bfU \bfSigma \bfV^*  = 
    \begin{bmatrix} 1 \\ 0 \\ \vdots \\ 0 \end{bmatrix}
    \begin{bmatrix}
        N\sqrt{N} 
    \end{bmatrix}
    \begin{bmatrix}
        \frac{1}{\sqrt{N}} & \frac{1}{\sqrt{N}} & \cdots & \frac{1}{\sqrt{N}} 
    \end{bmatrix}.
	\label{eq-single}
\eeq
Hence the singular values of $\mathring{\bfA}$ are 
$\sigma_1 = N\sqrt{N}$ and  $\sigma_2 = \cdots = \sigma_N = 0.$
In short, when SVD is applied to $\mathring{\bfA}$
rather than $\bfA$ there is only one nonzero singular value,
yielding a reduced basis with a single element $\{\bu_0\}$.
By shifting each snapshot by an appropriate number of grid cells (reversing the
transport due to the hyperbolic equation) they all line up. 
This procedure can be seen as a straightforward discretization of 
\eqref{eq-contimin}, and we formulate its generalization as follows.
\begin{defi}[Shift numbers]\label{def:tr}
    Let $\bfA \in \RR^{N \times M}$ be a real matrix and $\bb \in \RR^N$
    a real vector we will call the \emph{pivot}.
    Denote by $\ba_j$ the $j$-th column of $\bfA$, then define
    the integers $\nu_j \in \ZZ_N$ to be the minimizers
    \beq
    \nu_j = \argmin_{\omega \in \ZZ_N} 
    \left\lVert  \ba_j - \bfK^{\omega} \bb \right\rVert_2^2
    \Ffor j=1,2, ... , M.
    \label{eq-bnumin}
    \eeq
    Whenever the minimization is not unique, we choose one closest to 0.

    We call $\{\nu_j\}$  the \emph{shift numbers} and organize them in a vector 
    $\bnu \equiv (\nu_j)_{j=1}^M$. 
    We denote the computation of $\bnu$ in \eqref{eq-bnumin} as
    \beq
    \bnu = \cC (\bfA; \bb).
    \eeq
\end{defi}
In \eqref{eq-bnumin} we are merely
shifting the entries of the pivot $\bb$ to match $\ba_j$ as much as possible.
Here we introduce some more notations regarding the computation $\cC$.
\begin{notat} 
    Pivot operations.
    \begin{itemize}
        \item Let $\cC(\bfA ; j) \equiv \cC(\bfA; \ba_j)$,
            when the pivot is a column of $\bfA$.
        \item For $\bfB \in \RR^{N \times N}$, let 
            $(\cC(\bfA ; \bfB))_j \equiv \cC(\ba_j, \bb_j)$. 
            That is, in case $\bb$ in \eqref{eq-bnumin} depends on the column index 
            $j$ so that the pivot is allowed to change for each column, we 
            supply the matrix $\bfB \in \RR^{N \times M}$ to indicate that 
            its $j$-th column $\bb_j$ will be used as the pivot for 
            computing $\nu_j$.
        \item Given $\ell : \{1, ..., M\} \to \{1, ..., M\}$, 
            let $(\cC(\bfA; \ell))_j \equiv \cC(\ba_j; \ba_{\ell(j)})$. 
            We define a \emph{pivot map} $\ell$
            that designates the pivot for each column, and supply it to $\cC$.
    \end{itemize}
\end{notat}
The shift numbers $\bnu$ contain the information on how many entries
each columns of the matrix should be shifted. So $\bnu$ describes a transport
operation to be acted on each column, which will be summarized 
in the operator defined below.

\begin{defi}[Transport with periodic boundary conditions]
    Given a matrix $\bfA \in \RR^{N \times M}$ and a vector of shift numbers 
    $\bnu \in \RR^M$, the \emph{transport with periodic boundary conditions}
    $\cT$ is defined as
    \beq
        \cT(\bfA; \bnu) 
        \equiv \begin{bmatrix}
            \bfK^{\nu_1} & \bfK^{\nu_2} & \cdots & \bfK^{\nu_M}
        \end{bmatrix}
        \odot
        \begin{bmatrix}
            \ba_1 & \ba_2 & \cdots & \ba_M
        \end{bmatrix}.
    \eeq
    If a vector $\bb \in \RR^N$ is given instead of a matrix, we let
    \beq
        \cT(\bb; \bnu) 
        \equiv \begin{bmatrix}
            \bfK^{\nu_1}\bb & \bfK^{\nu_2}\bb & \cdots & \bfK^{\nu_M} \bb
        \end{bmatrix}.
    \eeq
\end{defi}
It is easy to see that $\cT(\cdot; \bnu)$ and $\cT(\cdot: -\bnu)$ 
are exact inverses of each other. That is, for fixed $\bnu$, 
\beq
\cT( \cT(\bfA; -\bnu) ; \bnu) = \bfA \Ffor \bfA \in \RR^{N \times M}.
\label{eq-Tinv}
\eeq

The key observation in the example above \eqref{eq-single} is that 
the SVD of $\cT(\bfA; -\bnu)$ with $\bnu = \cC(\bfA; \ba_1)$
results in faster decay in singular values than that of $\bfA$.
(The dynamics in $\cT(\bfA;-\bnu)$ represents the reduced dynamics $r(\tau)$
in symmetry reduction.)
Therefore one approximates $\cT(\bfA; -\bnu)$ by a low-rank
representation $\tilde{\bfA}$ via the usual truncation of rank-$1$ expansion. 
 If we apply the forward transport to $\tilde{\bfA}$,
that is, compute $\cT(\tilde{\bfA}; \bnu)$, it will be a better 
approximation of $\bfA$ compared to the direct low-rank approximation of $\bfA$.
This idea has been illustrated also in \cite{Schulze15}.
The effectiveness of this approach, along with extensions will be discussed
further in Section \ref{sec:extshift}.

Unfortunately, template fitting has several important drawbacks. 
We will demonstrate them through typical examples of hyperbolic PDEs. 
Suppose the given matrix $\bfA$ is a snapshot matrix
from the following four hyperbolic problems. Diagrams visually illustrating 
the solution behavior are shown in Figure \ref{fig-fourprobs}, with respective 
enumeration.

\begin{enumerate}
    \item[(P1)]\label{it-p3}
        Advection equation with source term,
        \beq
            u_t + u_x = - \gamma u \Iin (0,1),
            \quad \text{ with } \gamma > 0,\label{eq-advsrc}
        \eeq
        where $u(x,0)$ is a non-negative density pulse.
        The pulse diminishes in height over time,
        and this decrease cannot be well represented by translation
        alone. This is an inherent limitation of \eqref{eq-bnumin}. 
    \item[(P2)] \label{it-p2}
        Advection equation 
        \beq
        u_t + u_x = 0 \Iin (0,1), \label{eq-adv}
        \eeq
        with absorbing boundary conditions, $u_x = -u_t$ at 
        the right boundary $x = 1$, and $u(x,0)$ a density pulse.
        In \eqref{eq-bnumin} periodic shift $\bfK$ assumes periodic boundary
        conditions, so there is little hope of capturing this absorption.
    \item[(P3)] \label{it-p1}
        Acoustic equations in a homogeneous medium,
        \beq
        \begin{bmatrix} p \\ u  \end{bmatrix}_t
            +
        \begin{bmatrix} 0 & K_0 \\ 1/\rho_0 & 0 \end{bmatrix}
        \begin{bmatrix} p \\ u \end{bmatrix}_x = 0
            \Iin (0,1) \label{eq-ahm}
        \eeq
        with periodic boundary conditions and the initial conditions
        in which $p(x,0)$ is an acoustic pulse and $u(x,0) = 0$.
        For the state variable $p$, the initial profile splits into two, both scaled
        by half, and propagates in opposite directions. A single minimization
        problem \eqref{eq-bnumin} cannot be used to represent the
        two different speeds. 
    \item[(P4)] \label{it-p4}
        Burgers' equations
        \beq
            u_t + u u_x = 0 \Iin (0,1), \label{eq-burg}
        \eeq
        again with a density pulse as the initial condition.
        When the initial profile changes shape dramatically,
        translation alone cannot yield a good approximation. 
\end{enumerate}

In this section, we address these issues by the generalization of 
the operators $\cC$ and $\cT$, in which we add new features
to template fitting procedure \eqref{eq-bnumin}. 
Each of these features are introduced one by one in Sections \ref{ssec-proj}, 
\ref{ssec-cutoff}, \ref{ssec-greedy} and 
\ref{ssec-regular}. 
Algorithm \ref{alg-tr}
describes the final iterative algorithm. 

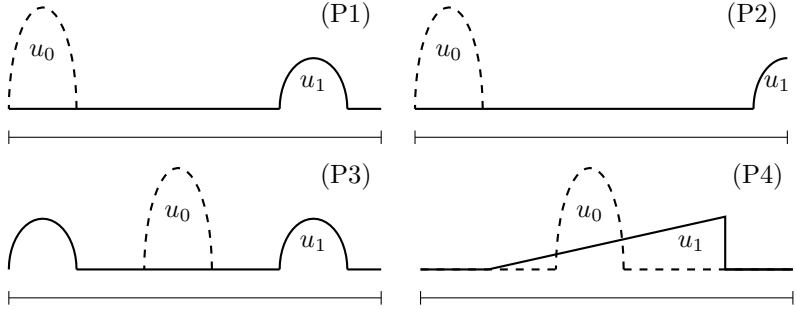
\begin{figure}
\centering
\begin{tabular}{cc}
\begin{tikzpicture}[scale=0.9]
    \draw (5.5,1.8) node[anchor=east]{(P1)};
\draw[black] (0.5, 1.25) node {$u_0$};
\draw (0,-.1) -- (0, .1);
\draw (0,0) -- (5.5,0);
\draw (5.5,-.1) -- (5.5, .1);
\draw[black,thick] (0.,.42) -- (4.,.42);
\draw[black,thick] (5.,.42) -- (5.5,.42);
\draw [black,dashed,thick,domain=0:180] 
plot ({0.5 + .5*cos(\x)}, {.42 + 1.5*sin(\x)});
\draw [black,thick,domain=0:180] plot ({4.5 + .5*cos(\x)}, {.42 + .75*sin(\x)});
    \draw[black] (4.5, .8) node {$u_1$};
\end{tikzpicture}
&
\begin{tikzpicture}[scale=0.9]
    \draw (5.5,1.8) node[anchor=east]{(P2)};
\draw[black] (0.5, 1.25) node {$u_0$};
\draw (0,-.1) -- (0, .1);
\draw (0,0) -- (5.5,0) ;
\draw (5.5,-.1) -- (5.5, .1);
\draw[black,thick] (0.,.42) -- (5.,.42);
\draw [black,dashed,thick,domain=0:180] 
plot ({0.5 + .5*cos(\x)}, {.42 + 1.5*sin(\x)});
\draw [black,thick,domain=90:180] plot ({5.5 + .5*cos(\x)}, {.42 + .75*sin(\x)});
\draw[black] (5.35, .8) node {$u_1$};
\end{tikzpicture}
\\ 
\begin{tikzpicture}[scale=0.9]
    \draw (5.5,1.8) node[anchor=east]{(P3)};
\draw[black] (2.5, 1.25) node {$u_0$};
\draw (0,-.1) -- (0, .1);
\draw (0,0)  -- (5.5,0) ;
\draw (5.5,-.1) -- (5.5, .1);
\draw[black,thick] (1.,.42) -- (4.,.42);
\draw[black,thick] (5.,.42) -- (5.5,.42);
\draw [black,dashed,thick,domain=0:180] plot ({2.5+ .5*cos(\x)}, {.42 + 1.5*sin(\x)});
\draw [black,thick,domain=0:180] plot ({4.5 + .5*cos(\x)}, {.42 + .75*sin(\x)});
\draw [black,thick,domain=0:180] plot ({0.5 + .5*cos(\x)}, {.42 + .75*sin(\x)});
\draw[black] (4.5, .8) node {$u_1$};
\end{tikzpicture}& 
\begin{tikzpicture}[scale=0.9]
    \draw (5.5,1.8) node[anchor=east]{(P4)};
\draw[black] (2.5, 1.25) node {$u_0$};
\draw (0,-.1) -- (0, .1);
\draw (0,0)  -- (5.5,0) ;
\draw (5.5,-.1) -- (5.5, .1);
\draw[black,thick,dashed] (0.,.42) -- (2.0,.42);
\draw[black,thick,dashed] (3.,.42) -- (5.5,.42);
\draw [black,dashed,thick,domain=0:180] plot ({2.5+ .5*cos(\x)}, {.42 + 1.5*sin(\x)});
\draw[black] (4.0, .8) node {$u_1$};
\draw[black,thick] (0.,.42) -- (1.0,.42);
\draw[black,thick] (4.5,.42) -- (5.5,.42);
\draw[black,thick] (1.0,.42) -- (4.5,1.2) -- (4.5,.42);
\draw[black,thick] (4.5,.42) -- (5.5,.42);
\end{tikzpicture}
\end{tabular}
\caption{Illustration of solution behavior for the four hyperbolic problems
(P1), (P2), (P3) and (P4). $u_0$ denotes the initial profile, drawn in dashed
lines, and $u_1$ denotes the solution at some future time, in solid lines.}
\label{fig-fourprobs}
\end{figure}

\subsection{Projection of pivot}\label{ssec-proj} Consider the situation in the problem (P1) above,
where the initial profile diminishes in height with time, while being transported at
constant speed. This is illustrated in Figure \ref{fig-fourprobs} (P1). 
The minimization problem in \eqref{eq-bnumin} does not take the scaling into 
account. We introduce a scaling by projecting the $j$-th 
column onto the pivot.
The projection $\cP$ is defined as
\beq
\cP(\ba_j ; \bb) \equiv 
\begin{cases}
    \dsp{\frac{\bb \bb^T}{\lVert \bb \rVert^2} \ba_j} & \text{ if } \lVert \bb \rVert > 0,\\
    \mathbf{0} & \text{ otherwise, if } \lVert \bb \rVert = 0. 
\end{cases}
\label{eq-projdef}
\eeq
Now, we replace the functional in the minimization problem \eqref{eq-bnumin}
by measuring the difference between the $j$-th column $\ba_j$ and the 
transported-and-projected vector $\cP(\ba_j ; \bfK^\omega \bb)$. 
That is, we solve the minimization problem
\beq
\nu_j = \argmin_{\omega \in \ZZ_N} \left\lVert
            \ba_j - \cP(\ba_j ; \bfK^{\omega} \bb) 
            \right \rVert_2^2.
            \label{eq-bnumin-proj}
\eeq
We denote this computation of the shift numbers in a concise form,
\beq
\bnu = \cC( \bfA; \bb, \cP),
\eeq
by supplying the projection map $\cP$.
The scaling \eqref{eq-projdef} must also be stored, and we organize it in the
vector $\bh$,
\beq
h_j = \cP(\ba_j; \bfK^{\nu_j} \bb)
\quad \text{ and } \quad
\bh = \begin{bmatrix} h_1 & \cdots & h_M \end{bmatrix}.
\eeq
We denote this concisely by writing $\bh = \cP(\bfA; \bb,\bnu)$.

The operator $\cT$ must also be generalized to take into account the 
scaling $\bh$,
\beq
\cT(\bfA; \bnu, \bh)
\equiv \begin{bmatrix} 
    h_1 \bfK^{\nu_1} & \cdots & h_M \bfK^{\nu_M}
    \end{bmatrix}
    \odot
    \begin{bmatrix}
        \ba_1 & \cdots & \ba_M
    \end{bmatrix}.
\eeq
The vector version of $\cT$ is modified similarly,
\beq
\cT(\bb; \bnu, \bh)
\equiv \begin{bmatrix} 
    h_1 \bfK^{\nu_1}\bb & 
    h_2 \bfK^{\nu_2}\bb &\cdots & 
    h_M \bfK^{\nu_M} \bb
    \end{bmatrix}.
\eeq

\subsection{Cut-off vectors}\label{ssec-cutoff} In (P2) we encounter a wave 
profile that is absorbed at the right boundary. 
If we were to apply the minimization problem with projection 
\eqref{eq-bnumin-proj}, the vanishing pulse would be partly represented  
by a translating profile that is decreasing in height. 
Still, some part of the profile will hit the boundary to the right,
and since $\bfK$ assumes a periodic boundary condition, 
the profile will appear at the left boundary as well. 
While one can modify $\bfK$ to account for this behavior, this will 
cause significant changes in the minimization problem \eqref{eq-bnumin},
and the existence of the exact inverse in \eqref{eq-Tinv} may be lost.
Moreover, another problem arises in the Burgers' equation (P4). 
The initial pulse is deformed to the extent that its transported profile 
may not represent the shape of the shock wave adequately, even with scaling.

As a step towards remedying both issues, we introduce the cut-off 
(or support) vector $\brho$.
Roughly speaking,
$\brho$ will designate the location of the domain where the projection
$\cP(\ba_j; \bfK^\omega \bb)$, for given $\omega \in \ZZ_N$,
is a good approximation of $\ba_j$.
We will denote by $\cS$ the operator that yields the cut-off, for given
two column vectors $\ba_j$ and $\bb$.
To be more specific, the $i$-th component $\rho_i$ of $\brho$ is
defined as
\beq
\rho_i =  \left( \cS(\ba_j ; \bb) \right)_i
\equiv 
\begin{cases}
    1 & \text{ if } \sign(a_{ij} - b_i) \cdot \sign(a_{ij}) \ge 0 \\
    & \quad \quad \quad 
    \text{ and } \abs{a_{ij} - b_i}
    \le \abs{a_{ij}}, \\
    0 & \text{ otherwise.}
\end{cases}
\label{eq-cS}
\eeq
The intention is to use $\rho_i b_i$ to approximate $a_i$.
The first condition $\sign(a_{ij} - b_i) \cdot \sign(a_{ij}) \ge 0$
ensures that the cut-off pivot does not overshoot the profile,
and the second condition $\abs{a_{ij} - b_i} \le \abs{a_{ij}}$ makes sure that 
the approximation has the same sign as the original vector. 

We will project $\ba_j$ onto $\bb$ for scaling before we apply $\cS$. 
That is, $\cP(\ba_j; \bb)$ will be input above in \eqref{eq-cS} in place of $\bb$. 
To simplify the notation, we will use the shorthand
\beq
\cS(\ba_j ; \bb, \cP) \equiv \cS(\ba_j; \cP(\ba_j; \bb)).
\eeq
Now the minimization \eqref{eq-bnumin-proj} is further updated:
we shift $\bb$ and scale it using the projection $\cP$, and we cut-off
using $\cS$, then we compare with $\ba_j$.
The new minimization problem becomes,
\beq
\nu_j 
= \argmin_{\omega \in \ZZ_N}
    \left \lVert \ba_j - 
    \cS(\ba_j ; \bfK^\omega \bb, \cP)
    \odot
    \cP(\ba_j; \bfK^\omega \bb) \right\rVert_2^2.
\label{eq-bnumin-cutoff}
\eeq
Here $\odot$ denotes the component-wise multiplication between two vectors in $\RR^N$.
As before, we define a shorter notation 
for this computation of shift numbers,
\beq
\bnu = \cC(\bfA; \bb, \cP, \cS).
\eeq
Furthermore, this generalization makes it necessary to store the vectors $\brho_j$
corresponding to each $\nu_j$, that is,
\beq
\brho_j = \cS(\ba_j ; \bfK^{\nu_j} \bb, \cP).
\eeq
We store these as columns of the matrix
\beq
\bfRho = \begin{bmatrix} \brho_1 & \cdots & \brho_M \end{bmatrix},
\eeq
and we also write $\bfRho = \cS(\bfA; \bb, \bnu, \cP)$.
In implementing the algorithm $\bfRho$ is computed simultaneously with $\bnu$,
but we will keep this implicit notation. $\bfRho$ can be stored as an 
array of Boolean data-type, so the storage requirement is not significant. 
It can be even more reduced should the pivot $\bb$ be 
sparse, but the details will not be pursued here.

Finally, the transport operator $\cT$ must also be extended to incorporate 
the cut-off function, which we can do by letting
\beq
\begin{aligned}
\cT(\bfA; \bnu, \bh, \bfRho)
&\equiv \begin{bmatrix} 
    h_1 \brho_1 \odot \bfK^{\nu_1} & \cdots & h_M \brho_M \odot \bfK^{\nu_M}
    \end{bmatrix}
    \odot
    \begin{bmatrix}
        \ba_1 & \cdots & \ba_M
    \end{bmatrix}\\
    &= \begin{bmatrix} 
    h_1 \brho_1 \odot \bfK^{\nu_1} \ba_1 & \cdots & 
    h_M\brho_M \odot \bfK^{\nu_M} \ba_M\end{bmatrix}.
\end{aligned}
\eeq
and the vector version $\cT(\bb; \bnu, \bh, \bfRho)$ is defined similarly.
Now we are ready to combine these computations in a greedy iteration.

\subsection{Greedy iteration and pivoting}\label{ssec-greedy}
Recall that the rank-1 expansion arising from the SVD can be seen as an 
iterative procedure in which a greedy rank-1 update is made in each iteration.
Here we define a similar update for the minimizations above, 
by attempting to capture transport structure iteratively. 
The necessity of multiple iterations
can be illustrated by the acoustic equation (P3). The initial pulse
splits into two and travels at two different speeds.
In this case the speeds have equal magnitude with opposite sign, 
as sketched in Figure 1 (P3).
However, they could be of the same sign and may also vary with time.
One minimization problem using $\cC$ and $\cT$ defined above cannot approximate
this behavior adequately.
Furthermore, in the Burgers' equation (P4) the profile is deformed heavily,
so that transporting one pivot once, even with projections and cut-offs, 
cannot capture the substantial change in shape. 

Therefore, we iterate on the previously defined computations as follows.
First, let $\bfR_1 \equiv \bfA$ and choose a pivot $\bb_1$, say 
the first column $\ba_1$ of $\bfA$. 
$\bfR_k$ will denote the residual, and index $k$ will be used 
for the iteration number.
We compute the shift numbers $\bnu_1$,
the scaling $\bh_1$ and the cut-offs $\bfRho_1$,
\beq
\bnu_1 \equiv \cC(\bfR_1; \bb_1, \cP, \cS) ,
\quad
\bh_1 \equiv \cP(\bfR_1; \bb_1, \bnu_1) 
\quad \text{ and } \quad
\bfRho_1 \equiv \cS(\bfR_1; \bb_1, \bnu, \cP).
\eeq
Now, we subtract off the first rough approximation from the snapshots,
\beq
\bfR_2 \equiv \bfR_1 - \cT(\bb_1; \bnu_1, \bh_1, \bfRho_1).
\eeq
This forms one iteration. 
We remark that the conditions in \eqref{eq-cS} prevent $\bfR_2$ from 
developing oscillations.

Next, we compute $\bnu_2, \bh_2$ and $\bfRho_2$ by 
replacing $\bfR_1$ above by $\bfR_2$.
We repeat, so that all transport patterns using the pivot $\bb_1$
are removed from the data. This reaches a point of diminishing return
after some iterations,
and we monitor the progress at the $k$-th iteration by computing the ratio 
$\lVert \bfR_k \rVert / \lVert \bfR_{k-1} \rVert$. One may set
a threshold $\tau_1$ so that 
\beq
\text{ if }
\frac{\lVert \bfR_k \rVert_F}{\lVert \bfR_{k-1} \rVert}_F > \tau_1,
\text{ then update the pivot } \bb_\ell \text{ to } \bb_{\ell+1}.
\eeq
There are many different options in choosing the next pivot $\bb_{\ell+1}$.
For example, one may proceed to a pivot that is orthogonal to the previous pivot.
Here we simply choose $\bb_\ell = \br_{\ell,k}$ where $\br_{\ell,k}$ 
is the $\ell$-th column of $\bfR_k$.

The algorithm halts when $\lVert \bfR_k \rVert_F < \tau_0$ for a
given tolerance $\tau_0$.

Let us organize the shift numbers $\bnu_k$ at each iteration in  $\bfV$,
\beq
\bfV \equiv \begin{bmatrix} \bnu_1 & \cdots & \bnu_K \end{bmatrix}.
\label{eq-bfV}
\eeq
where $K$ denotes the index of the last iteration.
We do the same for the cut-offs $\bfP_k$ and collect them in $\bfQ$,
\beq
\bfQ \equiv \begin{bmatrix} \bfP_1 & \cdots & \bfP_K \end{bmatrix},
\eeq
then similarly collect $\bh_k$ in $\bfH$,
\beq
\bfH \equiv \begin{bmatrix} \bh_1 & \cdots & \bh_K \end{bmatrix}.
\eeq

\subsection{Regularization of shift numbers} \label{ssec-regular}
The shift numbers $\bnu_k$ encode the transport motion of a profile 
over time, and we expect the speed of the transport to be relatively smooth. 
While the greedy iteration may yield a good approximation to $\bfA$,
the components of $\bnu_k$ may vary wildly.
Hence it is reasonable to enforce some regularity when computing $\bnu_k$.
That is, the shift number should change smoothly over time.
We achieve this by adding a penalty term in the minimization problem
\eqref{eq-bnumin-cutoff} when $j > 1$ to try to keep $|\nu_{(j-1)k}-\nu_{jk}|$
small:

\beq
\nu_{jk} = \argmin_{\omega \in \ZZ_N}
\left \lVert \ba_j - \cS(\ba_j ; \bfK^\omega \bb, \cP)
\odot \cP(\ba_j; \bfK^\omega \bb) \right \rVert_2^2
+ \lambda \abs{\omega - \nu_{(j-1)k}}^2,
\label{eq-bnumin-penalty}
\eeq
with a regularization parameter $\lambda$.
It may be desirable to add additional higher-order regularity terms, 
that is, second order finite difference term for $\bnu_k$.
Other penalty terms regarding the regularity of $\bfP_k$ as well 
as $\bh_k$ can be summed into \eqref{eq-bnumin-penalty} also.

The regularization is crucial, since the smooth evolution of
the shift numbers across snapshots is needed for 
\emph{displacement interpolation} in the sense used in optimal transport
(see, e.g., \cite{villani2008optimal}),
which effectively tracks the transport structure.
Note also that we may encode the smooth evolution efficiently
by polynomial interpolation or regression. 
This could be taken into consideration much earlier on, when 
the snapshots are generated: one may store snapshots at Chebyshev grid points
in the time variable, to facilitate accurate interpolation, then enforce high 
regularity in the shift numbers. 

The output of the regularized version of the algorithm can also be viewed 
as a greedy solution to an optimal transport problem, where one seeks to 
minimize the cost of transporting an initial state to the final state over 
admissible transport maps. The cost function here is particularly simple 
and is given by the regularization terms, for example the term penalizing
the total displacement in the case of \eqref{eq-bnumin-penalty}.

A simplified pseudo-code of the transport reversal is given in
Algorithm \ref{alg-tr}.

\begin{algorithm}[h]
\caption{Transport reversal algorithm}\label{alg-tr}
\begin{algorithmic}[1]
    \Procedure{TR}{$\bfA,K,\tau_0,\tau_1$}
    \Statex{}\Comment{input matrix $\bfA$, max. no. of iterations $K$, and tolerances $\tau_0, \tau_1$}
\State{$\ell \gets 1$}\Comment{pivot number}
\State{$k \gets 0$}\Comment{iteration count}
\State{$r_{\text{old}} \gets \Norm{\bfA}{F}$}
\State{$\bfR \gets \bfA$}\Comment{initialize residual $\bfR$}
\State{$\bb_\ell \gets \bfR(:,\ell)$}
\Comment{choose first column of $\bfR$ as pivot}
\While{($r_\text{old} > \tau_0 $ and $k \le K$)}
\State{$k \gets k + 1$}
\State{$(\bnu_k,\bh_k,\bfP_k) \gets (\cC(\bfR; \bb_\ell,\cP,\cS),
\cP(\bfR; \bb_\ell, \bnu_k),\cS(\bfR; \bb_\ell, \bnu_k,\cP)) $}
\Statex{}\Comment{computation is done concurrently}
\State{$\bfR \gets \bfR - \cT(\bb_\ell; \bnu_k, \bh_k, \bfRho_k)$}
\State{$r_\text{new} \gets \Norm{\bfR}{F}$}
\If{$r_\text{new} / r_\text{old} > \tau_1$}
\State{$\ell \gets \ell + 1$}\Comment{pivoting}
\State{$\bb_\ell \gets \bfR (:,\ell)$}\Comment{update pivot to be the 
                    $\ell$-th column of the new $\bfR$}
\EndIf
\State{$r_\text{old} \gets r_\text{new}$}
\EndWhile\label{euclidendwhile}
\State{$\bfB \gets [\bb_1, \cdots, \bb_\ell]$}
\State{$\bfV \gets [\bnu_1, \cdots, \bnu_k]$}
\State{$\bfH \gets [\bh_1, \cdots, \bh_k]$}
\State{$\bfQ \gets [\bfP_1, \cdots, \bfP_k]$}
\State{\textbf{return} $\bfB,\bfV, \bfH, \bfQ$}
\Statex{}\Comment{output pivots, shift numbers, scalings, and cut-offs}
\EndProcedure
\end{algorithmic}
\end{algorithm}

\subsection{Numerical example for transport reversal}

Here we apply the transport reversal algorithm to two of 
the problematic scenarios
given above: the acoustic equation (P3) and the Burgers' equation (P4).
The tolerances and regularization parameters are chosen rather
heuristically. $\lambda$ 
is set adaptively according to the variation of the functional 
in the minimization problem without the penalty terms \eqref{eq-bnumin-cutoff}:
we set $\lambda$ in \eqref{eq-bnumin-penalty} as
$2.5/(CN)$ where
\beq
\begin{aligned}
    C \equiv \max_{\omega \in \ZZ_N}  & \left \lVert \ba_j - \cS(\ba_j ; \bfK^\omega \bb, \cP)
\odot \cP(\ba_j; \bfK^\omega \bb) \right \rVert_2^2 \\
& \quad - \min_{\gamma \in \ZZ_N}\left \lVert \ba_j - \cS(\ba_j ; \bfK^\gamma \bb, \cP)
\odot \cP(\ba_j; \bfK^\gamma \bb) \right \rVert_2^2.
\end{aligned}
\eeq
The $L^2$-norm used for measuring the error here refers to the 2 dimensional
$L^2$-norm over spatial and temporal variables, 
$\lVert \cdot \rVert_F / \sqrt{NM}$ for a matrix in $\RR^{N \times M}$.

%
%

\begin{figure}
    \begin{tabular}{cc}
        \includegraphics[width=0.46\textwidth]{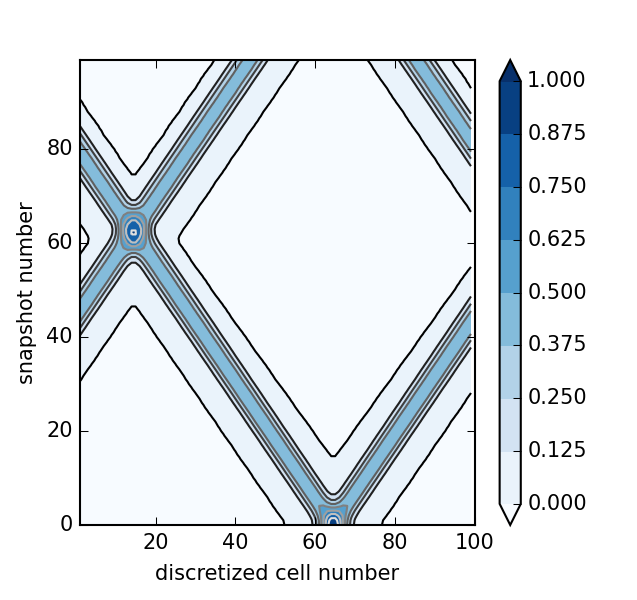} &
        \includegraphics[width=0.46\textwidth]{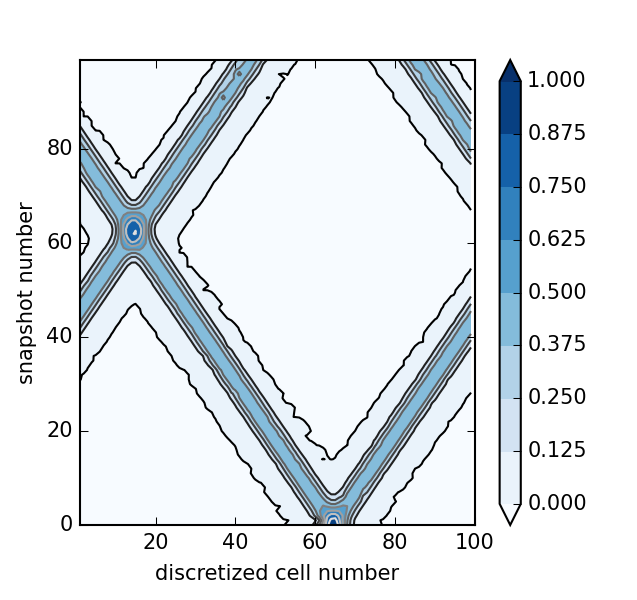} 
    \end{tabular}
    \caption{Snapshot matrix of the $p$ variable in the 
        acoustic equation (P3) (left) and 
    its approximation via the transport reversal algorithm (right). 
    The $L^2$-norm of the difference is $2.1841 \times 10^{-3}$.}
    \label{fig-snapshotvrecon-acoustic}
\end{figure}

\begin{figure}
    \includegraphics[width=1.0\textwidth]{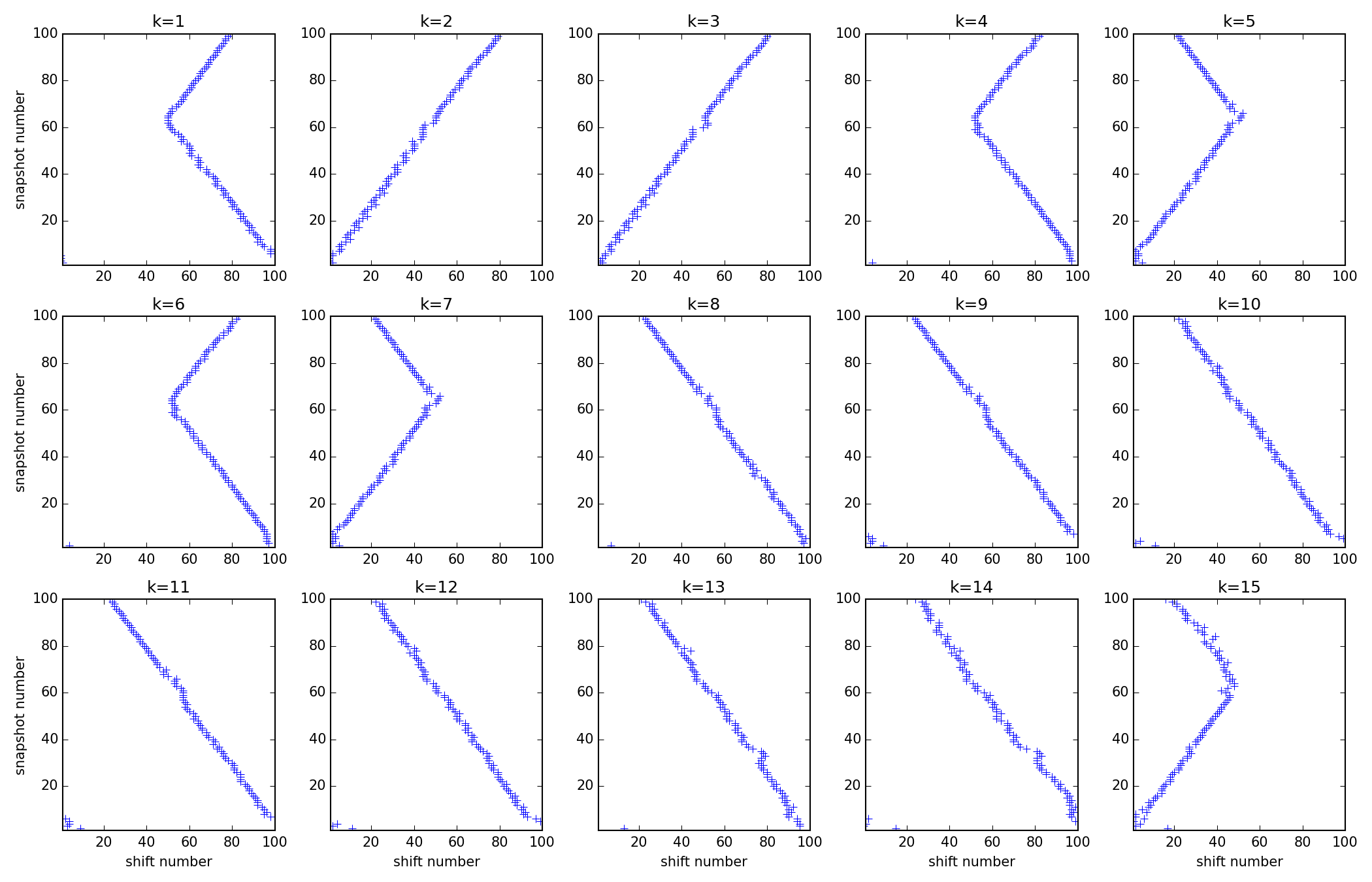}
    \caption{Shift numbers $\bnu_k$ \eqref{eq-bfV} for each iteration $k$,
        for the acoustic equations example. 
        Single pivot (the initial condition) was used.}
    \label{fig-courant-acoustic}
\end{figure}

%
%

\begin{figure}
    \begin{tabular}{cc}
        \includegraphics[width=0.46\textwidth]{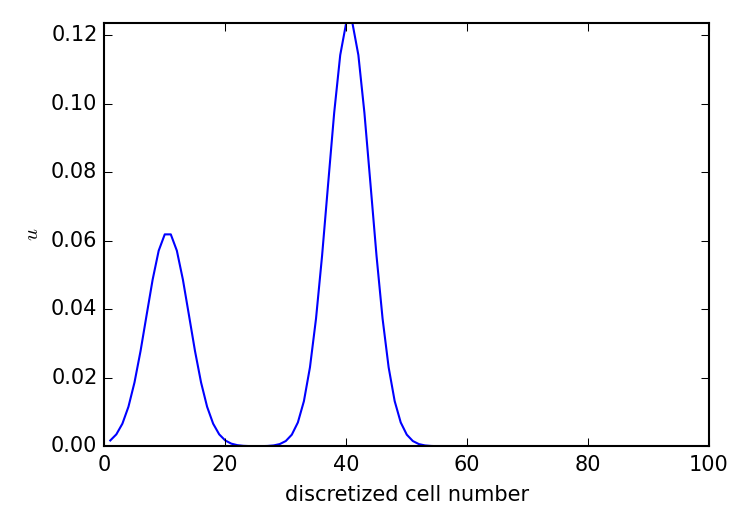} &
        \includegraphics[width=0.46\textwidth]{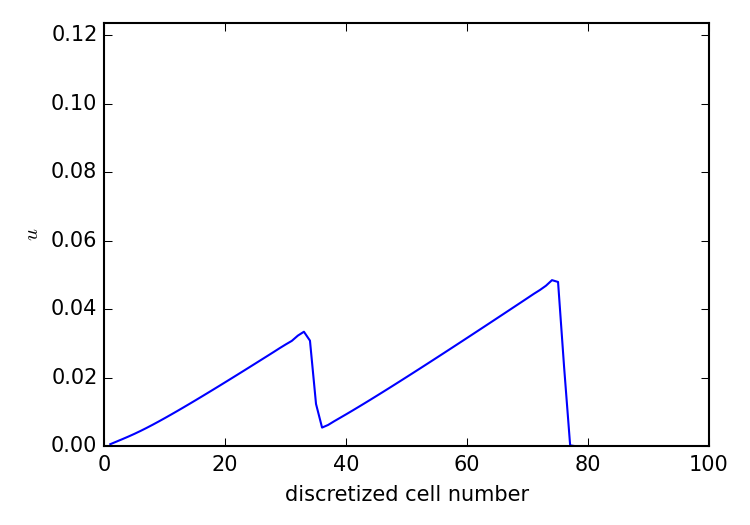} 
    \end{tabular}
    \caption{Initial condition of the Burgers' equation (P4) (left) and 
    its final snapshot (right).}
    \label{fig-initfinal-burgers}
\end{figure}

\begin{figure}
    \begin{tabular}{cc}
        \includegraphics[width=0.46\textwidth]{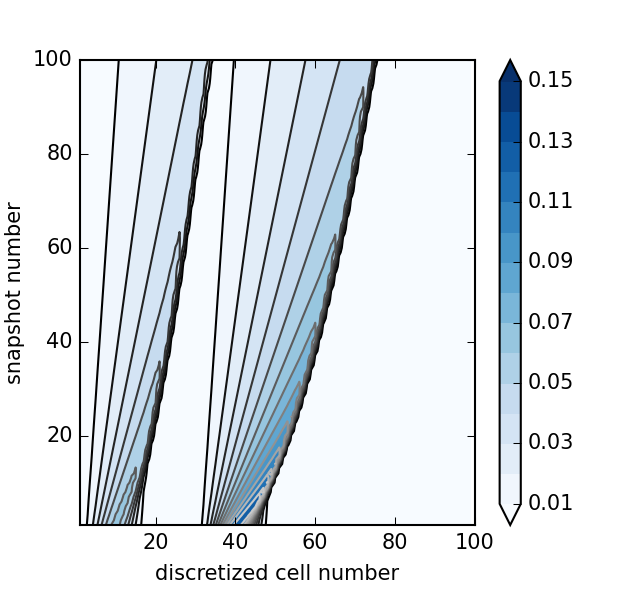} &
        \includegraphics[width=0.46\textwidth]{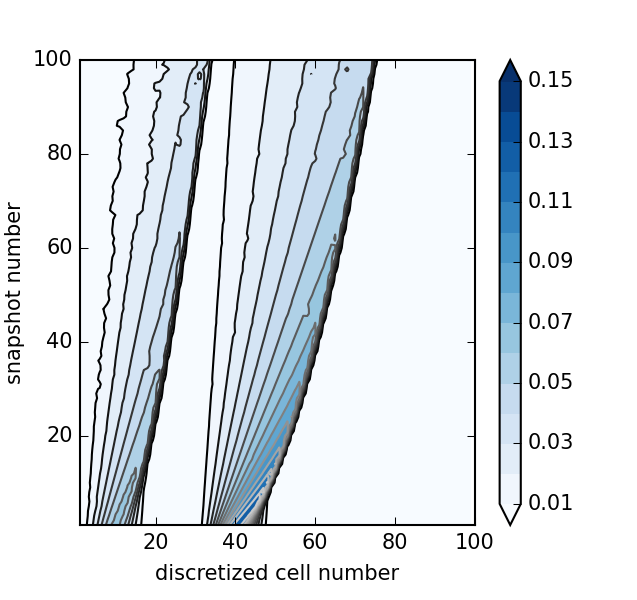} 
    \end{tabular}
    \caption{Snapshot matrix of the Burgers' equation (P4) (left) and 
    its approximation via the transport reversal algorithm (right).
    $L^2$-norm of the difference is $9.6333 \times 10^{-5}$.}
    \label{fig-snapshotvrecon-burgers}
\end{figure}

\begin{figure}
    \includegraphics[width=1.0\textwidth]{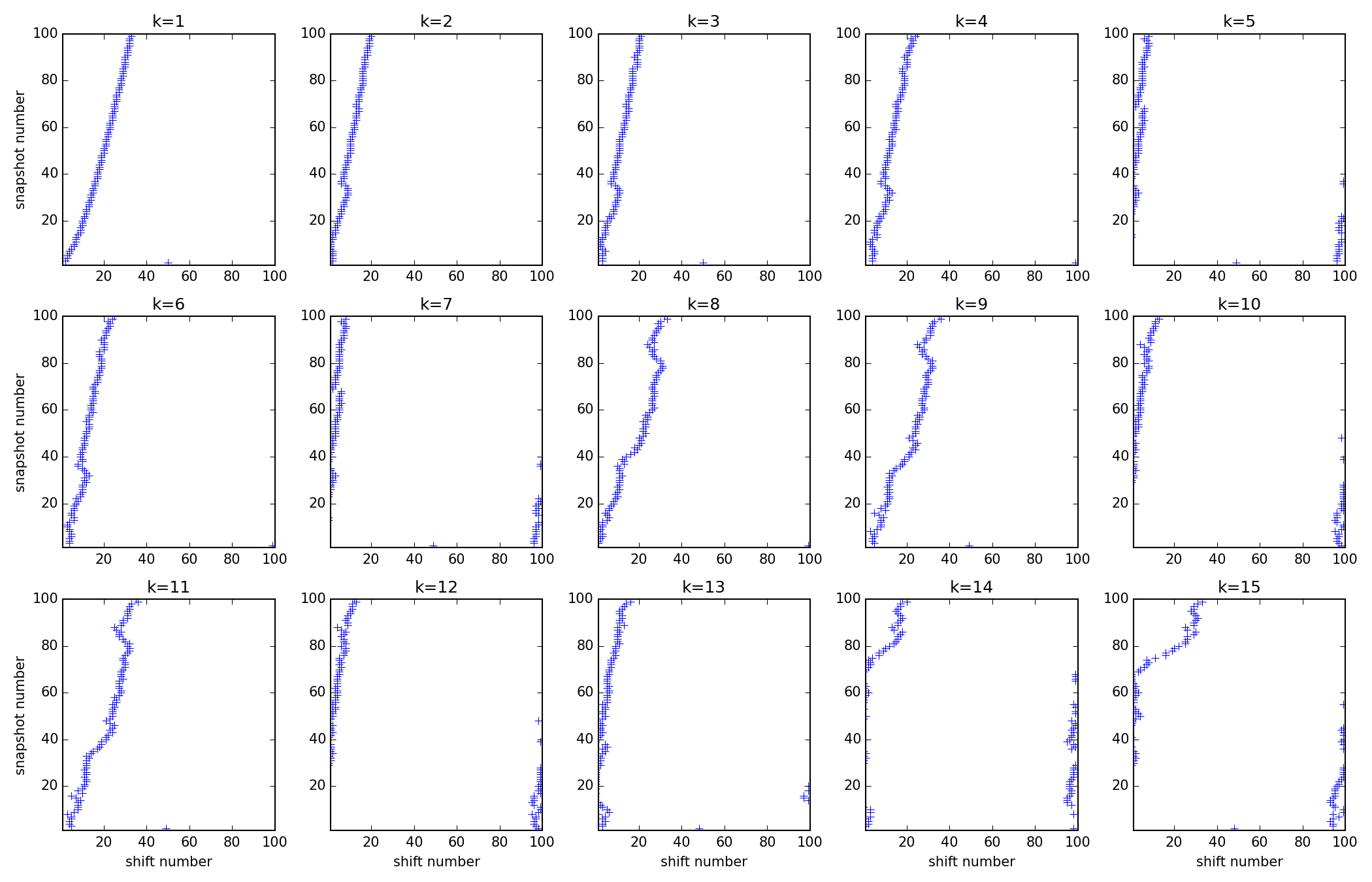}
    \caption{Shift numbers $\bnu_k$ \eqref{eq-bfV} for iterations $k = 1$ to
        $15$,
        for the Burgers' equations example. 
        The first pivot (the initial condition) was used for all
        iterations shown here.}
    \label{fig-courant-burgers}
\end{figure}

\begin{figure}
    \begin{tabular}{cc}
      Snapshot 15 & Snapshot 50 \\
      \includegraphics[width=0.46\textwidth]{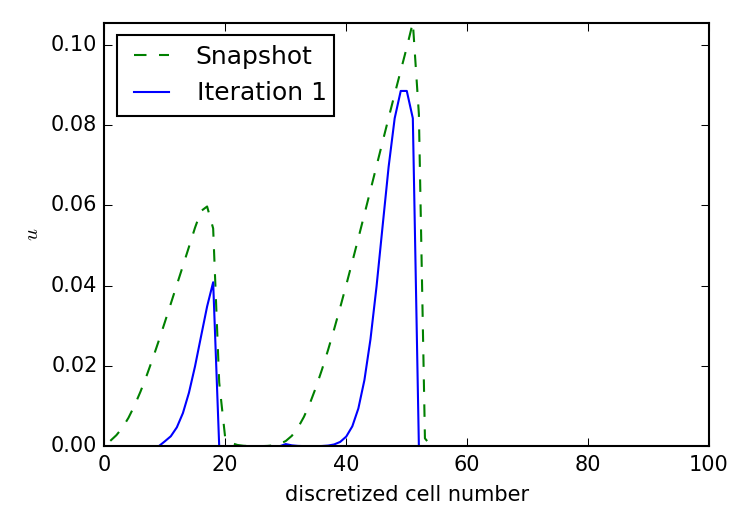} &
      \includegraphics[width=0.46\textwidth]{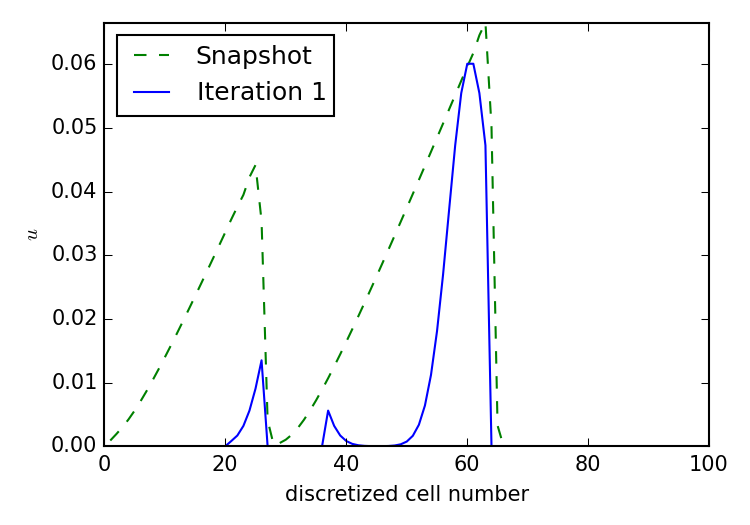} \\
      \includegraphics[width=0.46\textwidth]{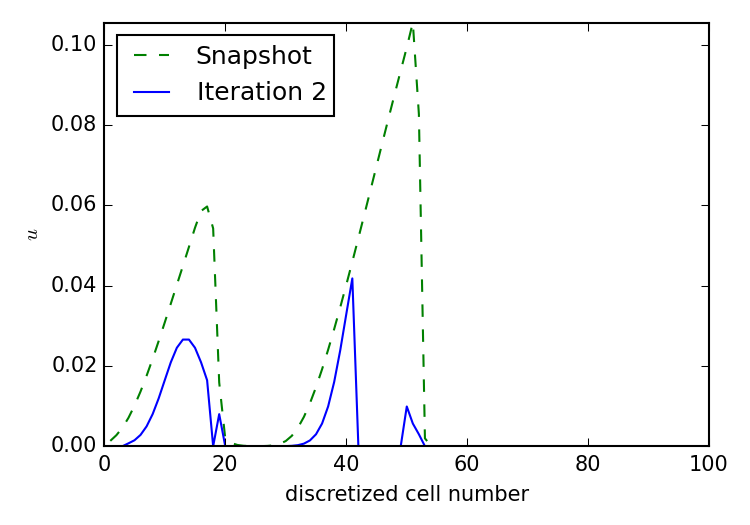} &
      \includegraphics[width=0.46\textwidth]{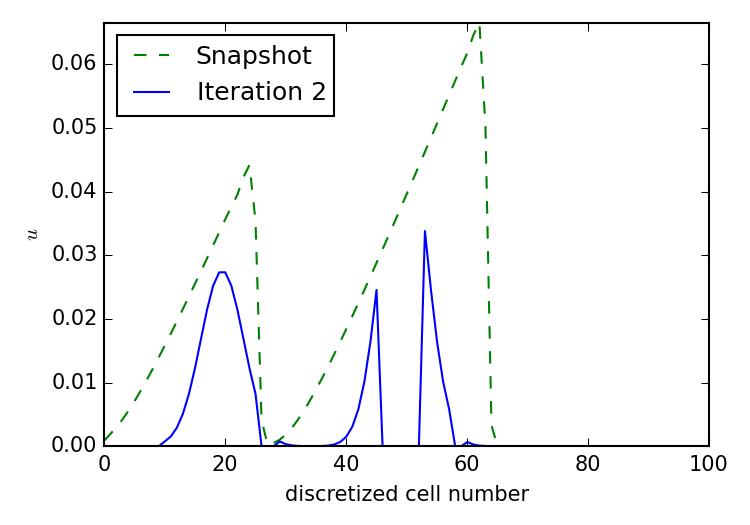} \\
      \includegraphics[width=0.46\textwidth]{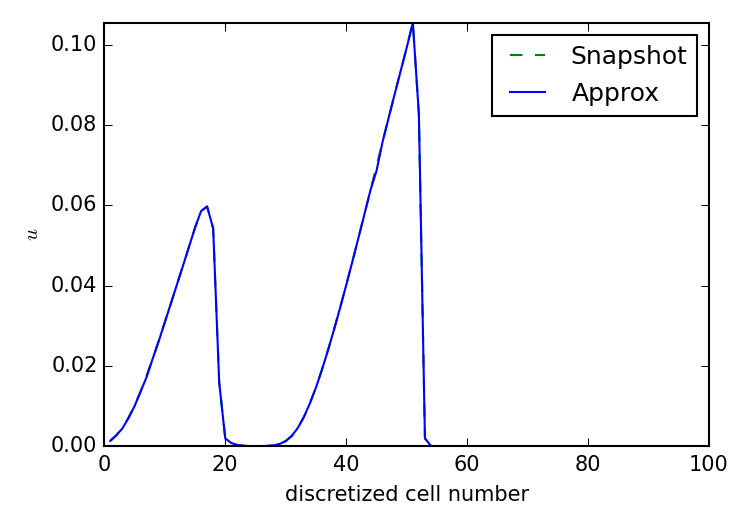} &
      \includegraphics[width=0.46\textwidth]{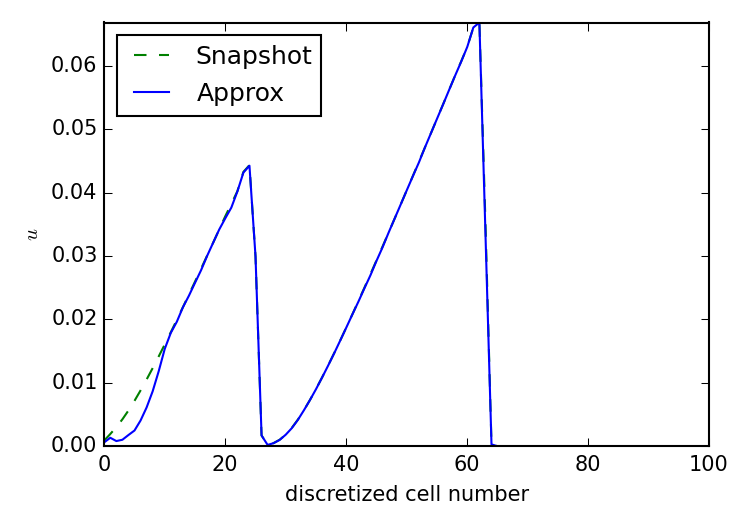} 
    \end{tabular}
    \caption{
    Contributions from first two iterations of the transport reversal
    (rows 1 and 2) for snapshots 15 and 50 (columns 1 and 2) 
    given by the expression \eqref{eq-contrib12}.
    Final approximation at iteration 30 is given in the last row.
    The contribution from first iteration alone
    attempts to capture the taller shock to the right (row 1), 
    and that of the second iteration captures the shorter one to 
    the left (row 2).
    }
    \label{fig-2pieces-burgers}
\end{figure}

\subsubsection{Acoustic equation} 
The snapshot of the $p$ variable for the acoustic equation is given in 
Figure \ref{fig-snapshotvrecon-acoustic}. No pivoting was required up to 
the given maximum number of iterations $K=15$. 
The $L^2$-norm of the residual at the final iteration 
was $2.1841 \times 10^{-3}$. The corresponding shift numbers for each iteration 
are shown in Figure \ref{fig-courant-acoustic}.
The diagonal pattern is clearly visible, which indicates that the method
is capturing the two profiles being transported at constant speed through
the domain. However, there is some ambiguity of between the two profiles when they
pass each other near snapshot number $60$. 
Notice how in the computed shift numbers for $k=1$ shown 
in Figure \ref{fig-courant-acoustic} the profile is transported first to the 
right, then the direction is reversed around the snapshot $60$, rather
than keeping straight. This behavior can be changed by adding
higher-order finite difference terms of $\bnu_k$ 
as penalty term in \eqref{eq-bnumin-penalty}
so that the second derivative of the shift numbers are kept small.

\subsubsection{Burgers' equation} 
Now we apply transport reversal to the snapshot matrix from 
the Burgers' equation (P4). The initial condition and its
final snapshot is shown in Figure \ref{fig-initfinal-burgers}.
The entire snapshot matrix and its approximate reconstruction are shown in 
Figure \ref{fig-snapshotvrecon-burgers}. 
The total number of iterations was $K=30$ and pivoting occured once
at iteration 19. 

The $L^2$-norm of the residual at the final iteration 
was $9.6333 \times 10^{-5}$. 
The corresponding shift numbers extracted for iterations 1-15 
are shown in Figure \ref{fig-courant-burgers}.
The shift numbers computed here also face an ambiguity between the two separate
humps at times, in a similar manner to the acoustic equations example. 
Adding more regularity will remove this ambiguity.

The first two shift numbers in Figure \ref{fig-courant-burgers}
correspond to the movement of the pivot that 
attempts to match the deforming hump to the left and to the right.
Since the left and the right humps move at slightly different speeds,
transporting the initial condition at constant speed is only able to match 
one of them.
It is helpful to isolate the contributions from the first two iterations from 
the algorithm. That is, we observe
\begin{equation}
    h_{jk} \brho_{jk} \odot \bfK^{\nu_{jk}} \ba_1
    \quad \text{ for }
    k = 1,2 \text{ and } j = 15, 50.
    \label{eq-contrib12}
\end{equation}
These contributions are shown in Figure \ref{fig-2pieces-burgers}.
Note how in the first iteration the initial profile is cut off to match
the hump to the right. In the second iteration, the hump to the right in the 
initial condition is cut off to match the left shock in the snapshot.
This illustrates the flexibility provided by the cut-off vectors $\brho$ 
for capturing the deformation occurring in the profile.

\section{Extensions of the shift operator} \label{sec:extshift}

In this section, we consider extensions of 
the matrix $\bfK$ \eqref{eq-ku0} above.
The matrix $\bfK^\omega$ ($\omega \in \ZZ_N$) is a basic component
of the minimization problem \eqref{eq-bnumin}, and two extensions of 
$\bfK^\omega$ will be given here.
First, we start with an extension using a linear interpolation
between $\bfK^\omega$ and $\bfK^{\omega + 1}$ by an analogy to the upwind flux.
This yields a continuous operator $\cK(\tilde \omega)$ 
over the real numbers ($\tilde \omega \in \RR$), rather than over integers.
Since this operator now causes some numerical diffusion due to its 
approximation, a reconstruction procedure is introduced to sharpen the profile.
The second extension allows the advection velocity in the advection 
equation \eqref{eq-advection} to depend on the spatial variable.
This extension is $\cK_c(\tilde \omega)$ with prescribed velocity field $c$.
We also discuss the pivoting procedure that becomes necessary
for linear systems.

We remark that the extensions that appear in this section can be used
in the greedy algorithm introduced in the previous section. 
Such use would only require that one change the operator $\bfK$ above
with $\cK$ or $\cK_c$.

\subsection{Extension of $\bfK$ by upwind flux}\label{sec:upwind}
First let us recall the finite volume upwind flux, which will motivate
our definitions below. The finite volume update of the advection equation \eqref{eq-advection} 
is given by
\beq
u^{n+1}_j = u^{n}_j - \frac{\Delta t}{\Delta x} 
    \prn{f^{n}_{j+1/2} - f^{n}_{j-1/2}}. \label{eq-upwind}
\eeq
The upwind flux is defined by $f^n_{j-1/2} \equiv c u_{j-1}^n$, and letting
$\nu \equiv c \Delta t/\Delta x$ be the \emph{shift number,}
the time-step \eqref{eq-upwind} 
can be is expressed as a linear interpolation between
$u_{j-1}^n$ and $u_j^n$,
\beq
u^{n+1}_j = \prn{\frac{c \Delta t}{\Delta x}} u^{n}_{j-1} + \prn{ 1 - \frac{c \Delta t}{\Delta x}} u^{n}_j 
= \nu u^n_{j-1} + \prn{1 - \nu} u^n_j.
\label{eq-fvmupdate}
\eeq
We write the update in \eqref{eq-upwind} as a matrix multiplication.

\begin{defi} \label{defi-upwindmat}
    Let $\bfK$ be the permutation matrix in \eqref{eq-ku0}.
    Define the matrix
    $\bfK(\nu) \equiv (1 - \nu) \bfI + \nu \bfK.$
Let us also define the \emph{discretized Laplacian} 
$\bfL_h \equiv (\bfK + \bfK^T - 2\bfI)/h^2,$
where $h = 1/N.$
\end{defi}
We list some basic properties of the matrix $\bfK(\nu)$.

\begin{restatable}{lem}{upwindbasic}\label{lem:upwind_basic}
$\bfK(\nu)$ satisfies
\begin{itemize}
\item[(a)] $\bfK(\nu) \bfK(\omega) = \bfK(\omega) \bfK(\nu)$.
\item[(b)] $\bfK(\nu) \bfK(\omega)^T = \bfK(\omega)^T \bfK(\nu)$.
\item[(c)] For $0 \le \nu,\omega \le 1$ and $\nu + \omega \le 1$, 
$\bfK(\nu) \bfK(\omega) = \bfK(\nu + \omega) + \cO\prn{1/N^2},$
where the constant for $\cO\prn{1/N^2}$ 
is a shifted discrete Laplacian (see paragraph below.)
\item[(d)] For $\bu \in \RR^{N}$, $\sum_{j=1}^N \prn{\bu}_j = \sum_{j=1}^N \prn{\bfK(\nu) \bu}_j$.
\end{itemize}
\end{restatable}

\begin{proof}
    The proof follows directly from definitions and 
    is given in Appendix \ref{ap:proof} \label{proof:upwind_basic}.
\end{proof}
The notation $\cO(1/N^2)$ here and throughout the paper is to be interpreted as
follows. When a matrix term is 
to be acted on vectors $\bv$ that are discretizations of twice differentiable 
functions on a grid of size $\cO(N)$ (so that its discrete Laplacian $\bfL_h \bv$ 
converges), the resulting term is of size $\cO(1/N^2)$.

Note that $\bfK(\nu)^T$ is obtained if we use the one-sided flux
in \eqref{eq-upwind} on advection with velocity $-1$ rather than $1$ (so that this is actually still the upwind flux).
Naturally, when the matrix $\bfK(\nu)^T \bfK(\nu)$ is multiplied
to the left of a vector, it propagates the entries of the vector 
first in one direction and then back in the opposite direction.
The resulting vector 
should be close to the initial one, in other words $\bfK(\nu)^T \bfK(\nu)$
must be close to the identity. This fact is summarized in the following lemma.

\begin{lem} \label{lem:laplacian}
    $\bfK(\nu)^T\bfK(\nu) = \bfK(\nu)\bfK(\nu)^T$ 
is an identity up to $\cO(1/N^2)$, in which the residual is a multiple of discrete Laplacian with periodic boundary condition.
It satisfies the bound
\beq
\Norm{ \bfK(\nu)^T \bfK(\nu) - \bfI}{2} \le 4\nu(1-\nu).
\label{eq-identity_error}
\eeq
\end{lem}

\begin{proof}
By definition and recalling $h= 1/N$, 
\beq
\bfK(\nu)^T \bfK(\nu) - \bfI
= \nu (1-\nu)(\bfK + \bfK^T - 2\bfI)
= \frac{\nu(1-\nu)}{N^2} \bfL_h,
\label{eq-ktk}
\eeq
Now, $\Norm{\bfK + \bfK^T - 2\bfI}{2} \le 4$ by von Neumann analysis 
\bals
\abs{\lambda_\xi} &= \abs{e^{i2\pi\xi (x + h)} - 2 e^{i2\pi\xi x} + e^{i 2\pi\xi(x-h)}} \\ 
&= \abs{e^{i2\pi\xi x} \prn{ e^{i 2\pi\xi h} - 2 + e^{-i 2\pi\xi h}}} 
= 2 \abs{\prn{\cos(2\pi\xi h) - 1}} \le 4.
\eals
\end{proof}

The inequality \eqref{eq-identity_error} holds for any $N$, and it is 
merely an estimate
for the total numerical diffusion due to the upwind flux, resulting from both
$\bfK(\nu)$ and $\bfK(\nu)^T$.
This marks a point of departure from the continuous setting
studied in \cite{marsden2, marsden1}, as it indicates that the 
translational actions discretized in such a way no longer form a 
group;
the inverse \eqref{eq-ktk} and multiplication (Lemma \ref{lem:upwind_basic} (c)
and Lemma \ref{lem:supwind_basic} (c)) are both only approximate and their residuals 
indicate the presence of numerical diffusion. 

Note that $\bfK(\nu)$ was defined in Definition \ref{defi-upwindmat} 
for any $\nu \in \RR$.
However, when $\nu$ is viewed as the Courant number,
the Courant-Friedrichs-Lewy condition imposes a necessary condition 
for stability of the upwind method \cite{fvmbook}, which 
requires $\nu \in [0,1]$ in this case.
This prohibits the use of $\bfK(\nu)$ when $\nu$ is 
outside of the unit interval,
at least superficially.
There is a straightforward generalization of this $\nu$ to be any real number 
$\tilde \nu \in \RR$ by first 
shifting exactly an integer number of times (determined by the integer part $s$ 
of $\tilde \nu$), which is 
accomplished by multiplying by $\bfK^s$,  and then applying 
$\bfK(\nu)$ where $\nu$ is the remaining fractional part.

This leads us to the next definition.
\begin{defi}
Given a shift number $\tilde{\nu} \in \RR$,
let $s$ and $\nu$ be its integral part and the fractional part, 
$s \equiv \floor{\tilde{\nu}}$ and $\nu \equiv \tilde{\nu} - s$, respectively.
We define the matrix $\cK(\tilde{\nu})$ as follows,
\beq
    \cK(\tilde{\nu}) \equiv 
\begin{cases}
\bfK^s\bfK(\nu) & \text{ if } s \ge 0 \\
\prn{\bfK^T}^s \bfK(\nu) & \text{ if } s < 0
\end{cases}.
\label{eq-supwindmat}
\eeq
We will also use the notation, for $s \in \ZZ$ and $\omega \in \RR$,
\beq
    \cK(s,\omega) \equiv 
\begin{cases}
\bfK^s\bfK(\omega) & \text{ if } s \ge 0 \\
\prn{\bfK^T}^s \bfK(\omega) & \text{ if } s < 0
\end{cases}.
\label{eq-supwindmat2}
\eeq
\end{defi}
Then it follows that $\cK(\tilde{\nu})^T = \cK(-\tilde{\nu})$
since if one writes out the integral and fractional parts $\tilde{\nu} = s + \nu$ and $-\tilde{\nu} = -(s+1) + (1-\nu)$,
\beq
\cK(\tilde{\nu})^T
= \prn{\bfK^T}^{s+1} \prn{(1- \nu)\bfI + \nu \bfK}
= \prn{\bfK^{s+1}}^T \bfK(1- \nu) = \cK(-\tilde{\nu}).
\label{eq-negation}
\eeq
All of Lemma \ref{lem:upwind_basic} follows through easily.

\begin{lem}\label{lem:supwind_basic}
    Let $\cK(\tilde{\nu})$ be as above. 
Let $s\equiv \floor{\tilde{\nu}}$, $\nu \equiv \tilde{\nu} - s$ and
$r\equiv \floor{\tilde{\omega}}$, $\omega \equiv \tilde{\omega} - r$.
Then it satisfies
\begin{itemize}
\item[(a)] $\cK(\tilde{\nu}) \cK(\tilde{\omega}) = \cK(\tilde{\omega}) \cK(\tilde{\nu})$.
\item[(b)] $\cK(\tilde{\nu}) \cK(\tilde{\omega})^T = \cK(\tilde{\omega})^T \cK(\tilde{\nu})$.
\item[(c)] For $\tilde{\nu}, \tilde{\omega}$ such that $0 \le \nu,\omega \le 1$ and $\nu + \omega \le 1$, 
we have $\cK(\tilde{\nu}) \cK(\tilde{\omega}) = \cK(\tilde{\nu} + \tilde{\omega}) + \cO\prn{1/N^2},$
where the constant for $\cO\prn{1/N^2}$ 
is a shifted discrete Laplacian.
\item[(d)] For $\bu \in \RR^{N}$, $\sum_{j=1}^N \prn{\bu}_j = \sum_{j=1}^N \prn{\cK(\tilde{\nu}) \bu}_j$.
\end{itemize}
\end{lem}
\begin{proof}
    We omit the proof, as it is similar to that of Lemma \ref{lem:upwind_basic}.
\end{proof}
This is a generalization of the upwind method beyond the
constraint of the CFL condition, and can be viewed as a
special case of large time-step (LTS) method
\cite{largetimestep1}, which can also be extended to nonlinear systems 
\cite{largetimestep3,largetimestep2}.

\subsection{Transport reversal in $\RR$}\label{sec:greedy}

Using the definitions above, we generalize the 
minimization problem \eqref{eq-bnumin} to be applied to a 
snapshot matrix $\bfA$.
We will use the notations
$\bfD \equiv \bfK - \bfI$ and 
$\bfL \equiv \bfK + \bfK^T - 2\bfI.$

\begin{restatable}{lem}{nulemma} \label{lem:nu} %
Suppose $\ba, \bb \in \RR^N$ are non-constant. 

If we let $\nu = \argmin_{\omega \in \RR} \Norm{\bb - \bfK(\omega)^T \ba}{2}^2$,
then
\beq
\nu = \frac{1}{2} \prn{ 1 - 2 \frac{ \ba^T \bfD \bb}{\ba^T \bfL \ba}}. \label{eq-nu}
\eeq
\end{restatable}

\begin{proof}
Let,
\bals
\cJ(\omega)\equiv \Norm{ \bb - \bfK(\omega)^T\ba}{2}^2
&=
\prn{ \bb - \bfK(\omega)^T\ba}^T\prn{ \bb - \bfK(\omega)^T\ba} \\
&= \bb^T \bb - 2\ba^T \bfK(\omega) \bb + \ba^T \bfK(\omega) \bfK(\omega)^T\ba .
\eals
Taking a derivative,
\[
    \cJ'(\omega) = - 2 \ba^T \bfK'(\omega) \bb + \ba^T \bfK'(\omega) \bfK(\omega)^T \ba + \ba^T \bfK(\omega) \bfK'(\omega)^T \ba.
\]
Letting $\cJ'(\nu) = 0$ and expanding, we have
\[
    \nu = \frac{1}{2}\prn{1 - 2\frac{ \ba^T ( \bfK - \bfI) \bb}{\ba^T ( \bfK + \bfK^T - 2 \bfI) \ba}}
   = \frac{1}{2}\prn{1 - 2\frac{ \ba^T \bfD \bb}{\ba^T \bfL \ba}} .
\]
The nullspace of $\bfL$ is the span of constant vectors, so the 
denominator on the RHS does not vanish.
\end{proof}
Recall that the CFL condition required that $\nu \in [0,1]$,
and note that $\nu$ given by \eqref{eq-nu} is not guaranteed
to lie in this stability region.
The minimization \eqref{eq-bnumin} is now extended to the case when $\bfK$
is replaced by the matrix $\cK(\tilde{\nu})$,
\beq
\min_{\tilde{\omega}\in \RR} \Norm{\ba - \cK(\tilde{\omega})^T\bb}{2}^2. 
\label{eq-gnumin}
\eeq
It is immediate that this problem is symmetric with respect to the
vectors $\ba$ and $\bb$ up to $\cO(1/N^2)$. That is, the problem 
can be rewritten using \eqref{eq-ktk},
\beq
\min_{\tilde{\omega}\in \RR} \prn{\Norm{\bb - \cK(\tilde{\omega})^T\ba}{2}^2 + \cO(1/N^2)}. \label{eq-gnusymmmin}
\eeq

The problem \eqref{eq-gnumin} 
is only of one variable $\tilde{\omega}$ lying in an 
interval $[0,N]$, although this can be viewed as a non-convex minimization problem 
in $\RR^N$
as we will see in Section \ref{sec:geo}.  
With the partitioning $\{[m,m+1]: m=0, \cdots, N-1\}$ 
of $[0,N]$, a recursive relation can be found for the formula
\eqref{eq-nu} in terms of $j$, restricting the variable $\tilde{\omega}$ to 
a set of $2N-1$ positive reals.

\begin{restatable}{lem}{recursion} \label{lem:recursion} %
Suppose $\ba,\bb \in \RR^N$ and let $\nu_s^\pm$ be defined as
\beq
    \nu_s^+ \equiv \argmin_{\omega \in \RR} \Norm{\bb - \cK(s,\omega)^T \ba}{2}^2,
    \quad
    \nu_s^- \equiv \argmin_{\omega \in \RR} \Norm{\bb - \cK(-s,\omega)^T \ba}{2}^2.
    \label{eq-gnu}
\eeq
Then we have the relations 
\beq
    \nu_{s+1}^{+} = \nu_s^{+} + \frac{(\bfK^s\bb)^T\bfL\ba}{\ba^T\bfL\ba},
    \quad
    \nu_{s+1}^{-} = \nu_s^{-} + \frac{\prn{(\bfK^T}^s\bb)^T\bfL\ba}{\ba^T\bfL\ba}.
    \label{eq-recursion}
\eeq
\end{restatable}
\begin{proof}
The proof is easy and is given in Appendix \ref{ap:proof}.
\end{proof}
Filtering out $\nu_s^\pm$ that do not satisfy the CFL condition, we let
\[
\hat{\nu}_s^\pm \equiv
\begin{cases}
\pm s + \nu_s^\pm & \text{ if } \nu_s \in [0,1]\\
\pm s & \text{ otherwise.}
\end{cases}
\]
Then the minimization problem \eqref{eq-gnumin} 
only requires comparison of at most $2N-1$ discrete values,
\beq
\tilde{\nu} =
\argmin_{\tilde{\omega} \in W} \Norm{\bb - \cK(\tilde{\omega})^T\ba}{2}^2
\text{ where } 
W \equiv \{0, \hat{\nu}^+_0, 1, \hat{\nu}^+_1, \cdots,N-1,\hat{\nu}^+_{N-1}\},
\label{eq-recursiongnu}
\eeq
or equivalently, 
$W = \{0,\hat{\nu}^-_{N-1}, 1, \hat{\nu}^-_{N-2}, \cdots 
 , N-1, \hat{\nu}^-_0\}.$

In many examples the data $\ba$ and $\bb$ have localized features.
This fact can be incorporated into our computation
of \eqref{eq-gnumin} by assuming that $\ba$ and $\bb$ are sparse
representations that reflect these features well, thereby $W$.
Reduction of $W$ beyond this may be possible by using discrete
Fourier transforms and exploiting isotropy that might exist in $\ba$
or $\bb$ (see Proposition \ref{prop:period}.)

\begin{defi}[Transport reversal in $\RR$]\label{def:reversal}
Given a matrix $\bfA \in \RR^{N \times M}$, 
let $\ba_j$ denote the $j$-th column of $\bfA$, 
and let $\bb \in \RR^N$ be a given \emph{pivot}.
Let 
\beq
\tilde{\nu}_j \equiv \argmin_{\tilde{\omega} \ge 0} 
\Norm{\ba_j - \cK(\tilde{\omega})^T\bb}{2}^2,  
\quad \text{ for } j=1, \cdots, M.
\label{eq-matrixgnu}
\eeq
This computation is denoted by $\tilde{\bnu} = \tilde{\cC}(\bfA; \bb)$.

We define the \emph{transport} $\cT$ of $\bfA$,
\beq
\cT(\bfA;\tilde{\bnu}) \equiv 
\begin{bmatrix} \cK(\tilde{\nu}_1) & 
\cK(\tilde{\nu}_2) & \cdots & \cK(\tilde{\nu}_{M}) \end{bmatrix} 
\odot 
\begin{bmatrix} \ba_1 & \ba_2  
& \cdots & \ba_{M} \end{bmatrix}. \label{eq-reversal}
\eeq
\end{defi}
Let $\tilde \nu \equiv \tilde \cC(\bfA; \bb)$ and 
$\mathring{\bfA} \equiv \cT(\bfA; - \tilde{\bnu})$.
Now, the orthogonality of the eigenvectors of
$\mathring{\bfA}^T \mathring{\bfA}$ is not strictly preserved
under the action of $\cK(\tilde{\nu})$, but holds up to $\cO(1/N^2)$. 
This is an analogue of Proposition 3 in \cite{titi} and is
stated as follows.

\begin{prop}\label{prop:ortho}
If $\varphi$ is an eigenvector of $\mathring{\bfA}^T \mathring{\bfA}$ 
with eigenvalue $\lambda$, then $\varphi$ is 
also an eigenvector of $\prn{\cK(\tilde{\nu})\mathring{\bfA}}^T \prn{\cK(\tilde{\nu})\mathring{\bfA}}$ to the same $\lambda$ for every
 $\tilde{\nu} \in \RR$, up to $\cO(1/N^2)$.
\end{prop}

\begin{proof}
Follows immediately from \eqref{eq-ktk} in Lemma \ref{lem:laplacian}.
\end{proof}

\subsection{Sharpening procedure} \label{sec:recon}

Once $\mathring{\bfA} \equiv \cT(\bfA; -\tilde{\bnu})$ is computed,
we can apply the SVD to 
construct a reduced basis representation of $\mathring{\bfA}$.
Let us denote this low-rank representation by $\tilde{\bfA}$,
and columns of $\tilde \bfA$ by $\tilde{\ba}_j$.
For a reconstruction of $\bfA$ itself, we compute the forward
transport, $\cT(\tilde{\bfA}; \tilde{\bnu})$.
While this yields an acceptable reconstruction, the numerical diffusion
arising from the upwind flux causes $\cO(1/N^2)$ amount of smearing.
This numerical diffusion has a particular structure 
\eqref{eq-ktk} in the form of a discrete Laplacian $\bfL_h$. 
This can be utilized to improve the accuracy by 
applying a post processing procedure motivated as follows.

Suppose we are given a column $\ba$, to which we apply the reversal
then reconstruction as above. Then the reconstruction,
which we denote by $\bb$, satisfies the equation
$\cK(\tilde{\nu}) \cK(\tilde{\nu})^T \ba = \bb.$
Recall that
\[
\cK(\tilde{\nu}) \cK(\tilde{\nu})^T = \bfI + \alpha \bfL_h
\where \alpha = \frac{\nu(1 - \nu)}{N^2}
\text{ and } 
\bfL_h \equiv \frac{1}{h^2} ( \bfK + \bfK^T - 2 \bfI ).
\]
Thus we can recover $\ba$ from $\bb$ by solving a discretized
Helmholtz equation augmented with a set of boundary conditions.
For example, we can use the first and last values of $\ba$,
\beq
    \prn{ \bfI + \alpha \bfL_h } \bu = \bb, 
	\quad \text{ satisfying } \quad
    u_1 = a_1 \quad \text{ and } \quad
	u_N = a_N.
\label{eq-dischelm}
\eeq
This inversion acts to remove the diffusive error caused by grid 
interpolation \eqref{eq-fvmupdate}. This can also be seen as a 
sharpening procedure, once rewritten as
\[
    \frac{ \bu - \bb}{k} = \beta \bfL_h \bu, \quad \beta \equiv 
    -\frac{\alpha}{k}.
\]
Due to the negative sign of $\beta$, 
here $\bu$ is shown as the single time-step solution to the backward heat 
equation with step size $k$ (a parameter that has been introduced for 
illustrative purpose).

We will denote this solution operator to \eqref{eq-dischelm} by
$\prn{ \bfI + \alpha \bfL_h }^{-1}.$
Letting $\alpha_j \equiv \nu_j(1 - \nu_j) / N^2$, we apply this
sharpening procedure for each column of the reconstructed $\tilde{\bfA}$,
that is,
\beq
\begin{bmatrix} \prn{ \bfI + \alpha_1 \bfL_h}^{-1} \cK(\tilde{\nu}_1) & 
\cdots & \prn{ \bfI + \alpha_{M-1} \bfL_h}^{-1}\cK(\tilde{\nu}_{M-1})
\end{bmatrix}
\odot
\begin{bmatrix} \tilde{\ba}_1 & \cdots & \tilde{\ba}_{M} \end{bmatrix}.
\label{eq-recon}
\eeq
As mentioned in remarks following Lemma \ref{lem:laplacian},
this procedure aims to address the fact that the discretized advection or translation no longer forms a symmetry group exactly. 
The reversal and reconstruction procedure will be demonstrated
numerically in Section \ref{sec:numerics}.

\subsection{Variable speed transport reversal and linear systems}\label{sec:varlin}

In the previous section we have introduced a transport procedure
amounting to a long-time solution of a constant speed advection equation.
Now we consider a generalization of the reversal problem 
\eqref{eq-contimin} when the advection speed $c > 0$ in \eqref{eq-advection} is 
allowed to depend on the spatial variable $x$. For simplicity,
$c$ will be represented as a piecewise constant function over 
a uniform grid.

In the previous section, the transport reversal \eqref{eq-matrixgnu}
has largely been a discretization of the 
continuous minimization problem \eqref{eq-contimin} with some numerical error
\eqref{eq-identity_error}.
But in considering the variable speed setting, additional differences between 
the discrete and the continuous case come to the fore. Consider $c$ which has the following 
property: 
there exists  $\overline{\omega} > 0$
and a grid $\{x_j\}_{j=0}^{N}$ and grid-sizes $\Delta x_j = x_{j+1} - x_j$,
\beq
    \abs{ \overline{\omega} - \omega_j } < \delta \ll 1
    \where
    \omega_j = \frac{c(x_{j+1/2})}{\Delta x_j}
        \Ffor j = 0, \cdots, N.
    \label{eq-easyvar}
\eeq
That is, even if $c(x_{j+1/2})$ varies, 
care can be taken to adjust size of the cells $\Delta x_j$ so that
the shift number  $\nu_j = \omega_j \Delta t$ behaves like a constant 
multiple of $\Delta t$ for all cells. Then, the discretized problem
is identical to the constant speed case and the techniques introduced
in the previous section apply directly, so the extension
to variable speed $c$ satisfying \eqref{eq-easyvar} is trivial.
Let us give a simple example of $c$ and $\{x_j\}_{j=0}^N$ that satisfies
this property. Consider the advection equation \eqref{eq-advection} and suppose
$c$ took on two values and the grid $\{x_j\}$ was constructed as follows,
\[
c(x) = 
\begin{cases}
 1 & \text{ if }0 \le x < \frac{1}{2} \\
 \frac{1}{2} & \text{ if } \frac{1}{2} \le x \le 1
\end{cases}, 
\Aand
 x_j = 
 \begin{cases}
  \frac{3}{2N}j & \text{ if } j < N/3\\
  \frac{3}{4N} j + \frac{1}{4}& \text{ if } j \ge N/3
 \end{cases}.
\]
If we choose $N$ to be a multiple of $3$, then \eqref{eq-easyvar} holds with 
$\bar{\omega} = 2N/3$ and $\delta = 0$.

However, the property \eqref{eq-easyvar} is not easily guaranteed, especially
when dealing with systems, when the characteristic variables have
different speeds, or when $c$ is allowed to depend on time, as in the nonlinear case.
Here we address the general variable speed case in 
Section \ref{sec:varspeed}, even if 
$c$ does not satisfy the property \eqref{eq-easyvar}.

\subsubsection{Variable speed reversal}\label{sec:varspeed}

We proceed to generalize the transport reversal procedure by
employing the large time-step (LTS) method 
\cite{largetimestep1,largetimestep3,largetimestep2}.
This is also reminiscent of Lagrangian methods such as the particle-in-cell method \cite{dawson,harlow} or its variant the material point method \cite{mpm2,mpm1}.
The LTS method allows 
long-time reversal of the given wave profile without
incurring excessive numerical diffusion,
mimicking the behavior of the matrix $\cK(\tilde \omega)$ .

The given vector $(u_j^n)_{j=1}^N$ will be considered to represent
a discretization of a function $u(x)$ lying on a uniform grid
$\{ x_j \}_{j=0}^{N}$ of the domain $\Dom = [0,1]$.
Let us define the jumps $\Delta_j^n \equiv u_{j}^{n} - u_{j-1}^n,$
where the index $j$ is defined modulo $N$.

Now, the grid points will serve as particles or material points, and their positions
will evolve with time.
 We index the time-dependence by $\ell$,
letting $\{x_j^\ell\}_{j=0}^N$ denote the grid points at time $t_\ell$.
Then we evolve the grid points as a function of 
time, $x_j = x_j(t)$, according to the ordinary
differential equation
\beq
\left\{
\begin{aligned}
\dot{x_j} = c(x_j), \\
x_j(0) = x_j^0,
\end{aligned}
\right.
\Ffor j = 0, 1, \cdots, N,
\label{eq-mpmode}
\eeq
with periodic boundary conditions.
We will evolve backward in time, as is natural for the reversal procedure.
The problem will be solved up to time $t^L < 0$.
The solution at $t^L$ is given by
\beq
    x_j(t^L) = x_j(0) + \int_0^{t^L} c\prn{x_j(t)} \dt .
	\label{eq-mpmode-sol}
\eeq
Recall $c(x)$ was assumed to be piecewise constant, so 
we let $c_j \equiv c(x_{j+1/2})$ and we 
compute \eqref{eq-mpmode-sol} explicitly. 

Let us be given time-steps $0=t^0 > t^1 > \cdots > t^L$ with 
$\Delta t^\ell = t_{\ell} - t_{\ell-1} < 0$. Define
$\Delta t_{ij}^\ell$ as the amount of time $x_j(t)$ lies in the $i$-th cell
$\cC_i$ during the time interval $[t_\ell, t_{\ell-1}]$,
$\Delta t_{ij}^\ell \equiv 
-\abs{\prnc{t \in [t_{\ell-1}, t_\ell]: x_j(t) \in \cC_i}}.$
Naturally, this is a partition of the time interval $[t_{\ell},t_{\ell-1}]$ so
$\Delta t^\ell = \sum_{i=0}^N \Delta t_{ij}^\ell.$
Then the solution $x_j^\ell$ at 
time $t_\ell = t_0 + \sum_{k=1}^\ell \Delta t^k$ is given by
\begin{equation}
x_j^{\ell+1} = x_j^\ell + \sum_{i=1}^N c_i \Delta t_{ij}^\ell
\quad \mod 1.
\end{equation}
Once time stepping has reached the final time $t^L$, 
we can update the cell average by computing the total flux for each cell,
\beq
u_j^{n+1} = u_j^n + \sum_{i=1}^N \Delta_i^n c_i \Delta t_{ij}^\ell .
\label{eq-ujn-update}
\eeq
The procedure is sketched in Figure \ref{fig-mpm-lt}.

\begin{figure}
\begin{tabular}{rr}
    \begin{minipage}{0.40\textwidth}
   Given $u_j^n$, place a material point at each grid point $x_j^0$ 
   (red dots) and 
   compute the jumps $\Delta_j^n$ at these points (red lines.) 
   Assign the jumps to the material points.
   \end{minipage}
   &
   \begin{minipage}{0.40\textwidth}
   Advect the material points, 
   computing the total flux caused by each jump, for each cell.
   (e.g, jump for $x_4^n$ below will change the volume of cells 
   $\cC_2, \cC_3, \cC_4$.)
       \end{minipage} \\
\includegraphics[width=0.45\textwidth]{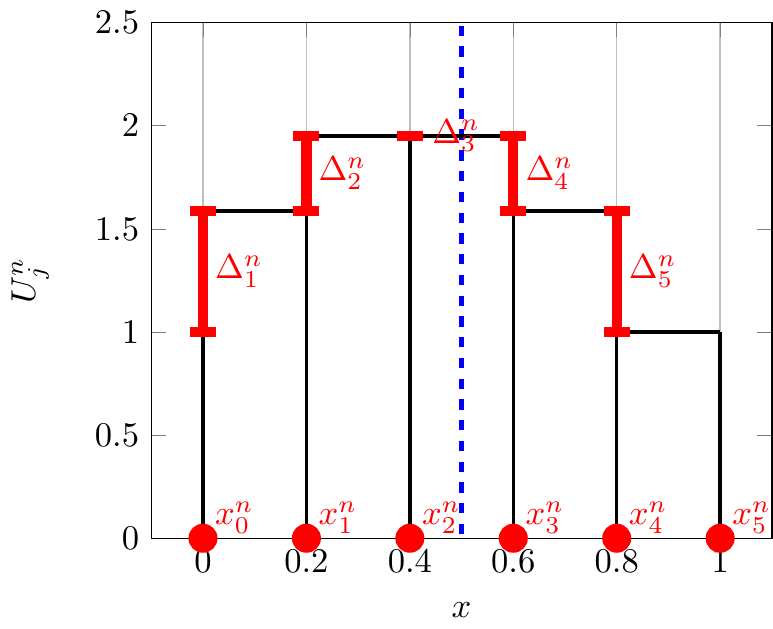} &
\includegraphics[width=0.45\textwidth]{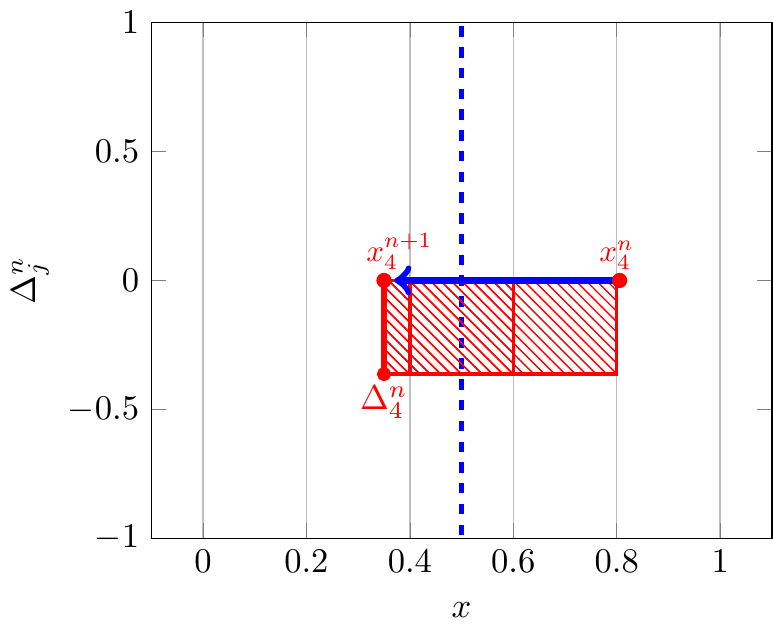} \\
 \end{tabular}
 \caption{An illustration of the variable speed advection
 with periodic boundary conditions. The dotted blue
 line at $x = 0.5$ denotes the interface where $c$ changes.
 After the fluxes for each cell is computed, we update $u_j^n$ by computing the
 total change of volume as in \eqref{eq-ujn-update}, resulting in $u_j^{n+1}$.
 }
 \label{fig-mpm-lt}
\end{figure}

Although the time steps 
$\Delta t^\ell$ are not technically necessary since this ODE
can be solved to any time in one step, we define it in order to maintain
an analogy to the constant speed case. 
Let $T$ denote the period of the solution $x_j$ to \eqref{eq-mpmode}, so that $x_j(t+ T) = x_j(t)$.
Recall that shift number $\tilde{\nu} \in \RR$
satisfied the periodicity condition 
$\cK(\tilde{\nu} + N) = \cK(\tilde{\nu})$, 
which should correspond to periodicity $x_j(t + T) = x_j(t)$.
Notice that $T$ satisfies the relationship
\beq
\bar{c} \equiv \frac{\abs{\Omega}}{T} = \frac{1}{\abs{\Omega}} \int_\Omega c(x) \dx ,
\eeq
where $\abs{\Omega}$ denotes the measure of $\Omega$.
Given $\tilde{\nu}$ and $c$  we let the time step $\Delta t$ satisfy
\[
    \tilde{\nu} = \frac{N\Delta t}{T} 
	= \bar{c}\frac{\Delta t}{\abs{\Omega}/N}
    = \prn{ \sum_{j=0}^N c_j \Delta x_j^0}
    \frac{\Delta t}{\sum_{j=0}^N \Delta x_j / N}.
\]

Using this definition, we denote the reversal 
procedure above by the operator $\cK_c(\tilde{\nu})$.
Note in particular that 
$\cK_c(\tilde{\nu} + N) = \cK_c(\tilde{\nu})$
holds with high numerical accuracy.
Now the minimization problem to be used in the variable speed reversal
can be written down. Replacing the transpose in \eqref{eq-gnumin} by a negation
of the argument using \eqref{eq-negation}, we have
\beq
    \min_{\tilde{\omega} \in \RR}
    \Norm{\ba_j - \cK_c(-\tilde{\omega})\bb }{2}^2.
	\label{eq-varmin}
\eeq
Now we define the reversal for the variable speed case.

\begin{defi}[Variable speed transport reversal]\label{def:varreversal}
Given $\bfA \in \RR^{N \times M}$
and a variable speed $c:[0,N] \to \RR$,
let $\ba_j$ denote the $j$-th column of $\bfA$. 
Let $\bb \in \RR^N$ be a given \emph{pivot}.
Then compute
\beq
\tilde{\nu}_j = \argmin_{\tilde{\omega}\in \RR} 
\Norm{\ba_j - \cK_c(-\tilde{\omega})\bb}{2}^2,  
\quad \text{ for } j= 1, \cdots, M.
\label{eq-vargnu}
\eeq
This computation is denoted by $\tilde{\bnu} = \tilde{\cC}_c(\bfA; \bb)$
where $(\tilde{\bnu})_j = \tilde{\nu}_j$.

We define the \emph{transport} of $\bfA$ with speed $c$, denoted by
\beq
\cT_c(\bfA; \bnu) \equiv 
\begin{bmatrix} \cK_c(\tilde{\nu}_1) & 
 \cdots & \cK_c(\tilde{\nu}_M) \end{bmatrix} \odot 
\begin{bmatrix} \ba_1  & \cdots & \ba_{M} \end{bmatrix}. \label{eq-vartransport}
\eeq
Also define the \emph{transport reversal} of $\bfA$, distinguished by the
sign of the shift numbers $\tilde{\bnu}$,
\beq
\cT_c(\bfA;-\tilde{\bnu}) \equiv 
\begin{bmatrix} \cK_c(-\tilde{\nu}_1) & 
 \cdots & \cK_c(-\tilde{\nu}_M) \end{bmatrix} \odot 
\begin{bmatrix} \ba_1 
& \cdots & \ba_{M} \end{bmatrix}. \label{eq-varreversal}
\eeq

\end{defi}

The orthogonality condition in Proposition \ref{prop:ortho} still holds
with small error.
Also, no simple relation such as \eqref{eq-recursion}
are found, and the sharpening procedure \eqref{eq-recon} cannot
be easily applied. This is due to the loss of convexity
to be discussed in Section \ref{sec:geo}.

\subsubsection{Reversal for linear systems with pivoting}\label{sec:pivoting}
Let us now consider the reversal for snapshot matrices
arising from linear systems of equations. We will focus on the 
acoustic equation (P4), 
\beq
\begin{bmatrix} p \\ u \end{bmatrix}_t
+ \begin{bmatrix} 0 & K \\ 1/\rho & 0 \end{bmatrix} 
\begin{bmatrix} p \\ u \end{bmatrix}_x= 0. \label{eq-acousticfull}
\eeq
Parameters $\rho$ and $K$ are the density and the bulk modulus of 
compressibility of the material, respectively.
Eigendecomposition of the matrix yields eigenpairs,
\beq
\lambda^1 = u -c, \quad 
\br_1 = \begin{bmatrix} - \rho c \\ 1 \end{bmatrix} \Aand
\lambda^2 = u + c,
\quad \br_2 = \begin{bmatrix} \rho c \\ 1 \end{bmatrix}.
\label{eq-eigenvec}
\eeq
where $c \equiv \sqrt{K/\rho}$.
We can rewrite the equation \eqref{eq-acousticfull} in terms
of new variables $r_1$ and $r_2$ by projecting the
state vector $[p,u]^T$ onto the eigenspace spanned by the
two vectors in \eqref{eq-eigenvec}.
When $c$ is constant, the system can be completely decoupled, and
two advection equations can be solved separately.
However, when $c$ depends on the spatial variable, the eigendecomposition
also depends on $x$. This implies that even after the eigendecomposition, there is a coupling between the variables $r_1$ and $r_2$
across space if $\rho c$ varies, so that the wave profiles will evolve. 
For example, $r_1$ may initially be identically zero at initial time but 
suddenly 
develop nonempty support as soon as a wave profile in $r_2$ passes through 
an interface and is partially reflected. An example of this kind is shown in 
Figure \ref{fig:acoustic_htgs_recon} below. 

Therefore we need to dynamically change the pivot vector $\bb$ in 
\eqref{eq-varmin} appropriately to other columns as the wave profile evolves. 
Let us recall the \emph{pivot map} $\ell: \{1, \cdots, M\}
\to \{1, \cdots, M\}$ which takes each column $\ba_j$ to its
corresponding pivot $\ba_{\ell(j)}$. 
Then we may define reversal with pivoting as follows.

\begin{defi}[Variable speed transport reversal with pivoting]
    \label{def:reversalpivot}
Let the matrix $\bfA \in \RR^{N \times M}$ be given,
a pivot map $\ell:\{1, ..., M\} \to \{1, ... , M\}$,
and a variable speed $c:[0,N] \to \RR$.
Then let
\beq
\tilde{\nu}_j \equiv \argmin_{\tilde{\omega}\in \RR} 
\Norm{\ba_j - \cK_c(-\tilde{\omega})\ba_{\ell(j)}}{2}^2,  
\quad \text{ for } j= 1, \cdots, M.
\label{eq-pivotgnu}
\eeq
This computation is denoted by $\tilde{\bnu} = \tilde{\cC}_c(\bfA; \ell)$
where $(\tilde{\bnu})_j = \tilde{\nu}_j$.
\end{defi}

The proper pivoting criterion will depend on the problem at hand, and for 
acoustic equations with heterogeneous media, pivoting when
there is large relative change in the $\ell^2$-norm difference
between the previous and current column was sufficient.
See Example \ref{expl:htgs} for numerical results using this
particular pivoting criterion.

\subsubsection{Numerical experiments} 
\label{sec:numerics}

We apply the transport reversal and reconstruction procedure outlined in this 
section to the acoustic equation, in both homogeneous and heterogeneous 
media. We do not introduce the iterative procedure from Section \ref{sec:rev}, 
and consider relatively simple examples to focus on the effect of
the extensions of the shift operator.

\subsection{Acoustic equation in homogeneous media}
\label{expl:homogeneous}

We apply the reversal \eqref{eq-reversal} to the constant speed acoustic 
equation \eqref{eq-acousticfull} with $K \equiv K_0$ and $\rho \equiv \rho_0$
with periodic boundary conditions. 
For the initial conditions, $p_0$ is given to be a Gaussian hump and $u_0$ 
to be identically zero.

The $100$ snapshots were taken from a $100$-cell solution. After an
eigendecomposition of the state vectors \eqref{eq-eigenvec} we 
transform the state variables $p$ and $u$, to
the characteristic variables $r_1$ and $r_2$.
Then we apply the reversal procedure \eqref{eq-reversal} to the two
snapshot matrices corresponding to these variables.

The decay of the singular values for $\mathring{\bfA}$ is clearly much 
more rapid, as seen from Figure \ref{fig:acoustic_singvals}.
The threshold of $99\%$ is achieved with only $3$ basis vectors.
The reconstruction is plotted against the snapshot itself
in Figure \ref{fig:sfvr_recon}, 
and they are nearly identical.
The POD reconstruction is also plotted.
The $L^2$-errors from the two reconstructions are compared in 
Figure \ref{fig:acoustic_singvals}, where the reversal consistently 
outperforms the na\"ive POD.

\begin{figure}
	\centering
    \begin{tabular}{cc}
   \includegraphics[width=0.45\textwidth]{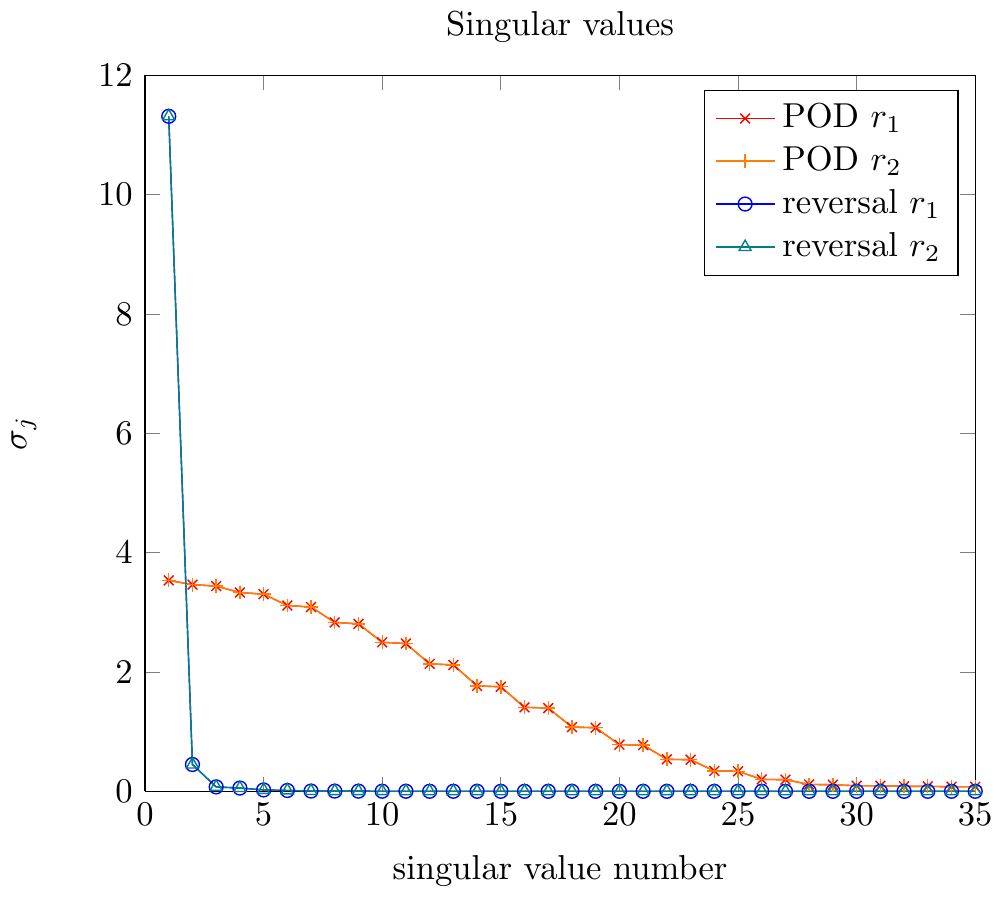} &
 \includegraphics[width=0.45\textwidth]{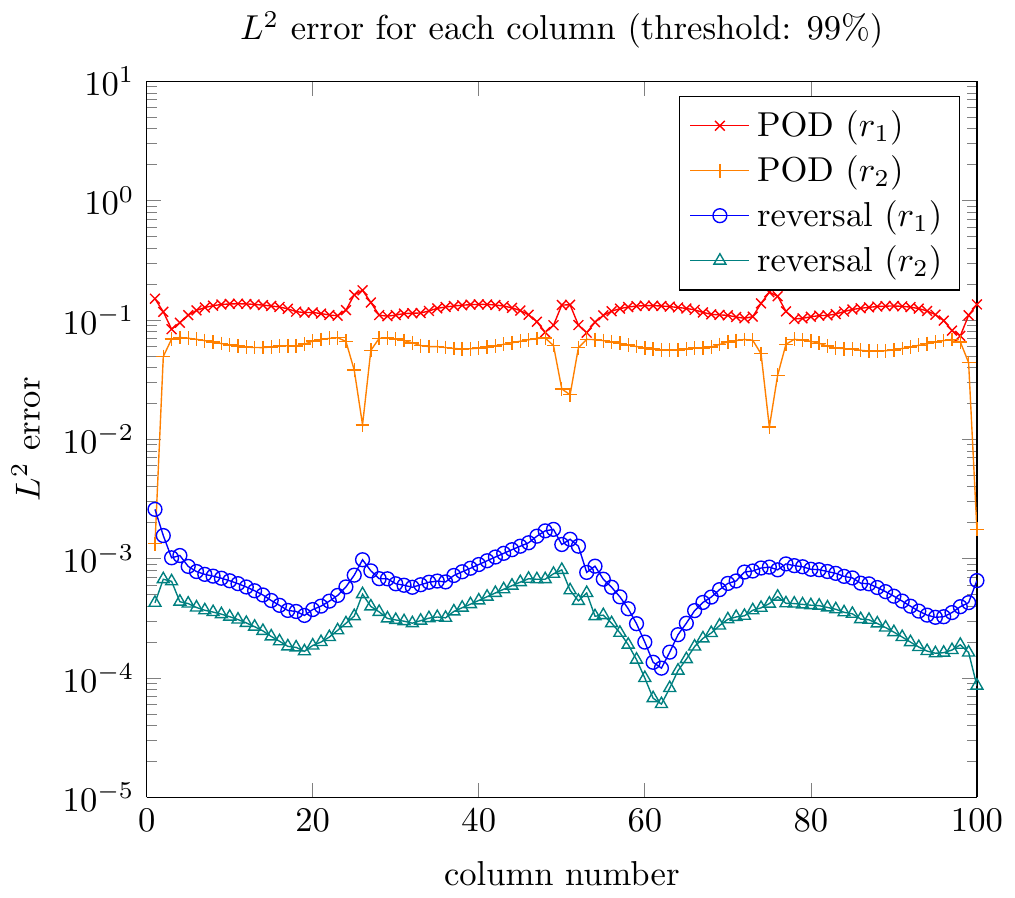} 
\end{tabular}
\caption{Left: Fast decay of singular values of snapshots in the variables $r_1$
and $r_2$ for the acoustic equation with homogeneous media. 
The largest $35$ singular values are shown. Three singular values represent $99\%$ 
of the threshold for both variables.
Right: $L^2$ error for each column of the reconstruction for eigenvector 
variables $r_1$and $r_2$.}\label{fig:acoustic_singvals}
\end{figure}

\begin{figure}
    \centering
  \includegraphics[width=0.8\textwidth]{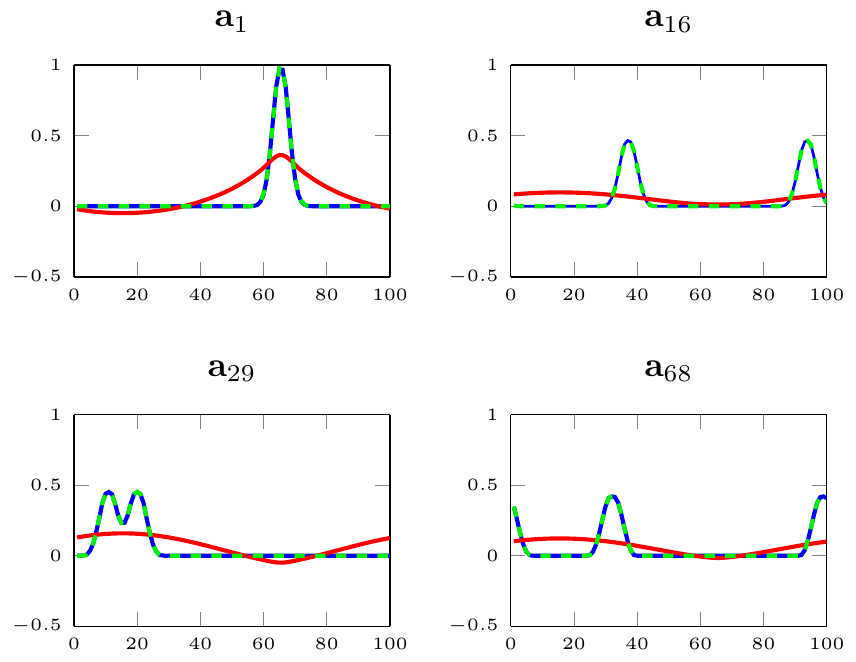}
    \caption{Reconstruction of the solution to acoustic equation
     with homogeneous media, for state variable $p$.
    The snapshot is given in dashed green, the reversal reconstruction in blue,
    and POD reconstruction in red. 
    Three reduced basis vectors were used for both reversal and POD.
    }
    \label{fig:sfvr_recon}
\end{figure}

\subsection{Acoustic equation with heterogeneous media}
\label{expl:htgs}

Now we consider the acoustic equation \eqref{eq-acousticfull} with heterogeneous
media, with two different materials. An interface will be located at $x=0.5$.  Letting $\ell$ designate \emph{left} part of the domain $(0,0.5)$ and $r$ the \emph{right} part of the domain $(0.5,1),$ suppose we have the parameters
$K$ and $\rho$ vary depending on the part of the domain. Here we let
$\rho_\ell = 1$,
$K_\ell = 1$ and
$\rho_r = 4$,
$K_r = 1$.
so that $c_\ell =1$ and $c_r = 0.5$.
We again impose periodic boundary conditions, and this creates
two more interfaces, at $x = 0$ and $1$. 
The initial condition $p_0$ and $u_0$ are both Gaussian 
humps of identical shape traveling towards the interface at $x = 0.5$.

The $100$ snapshots were taken from a $100$-cell solution. An
eigendecomposition of the state-vectors transform the variables
$u$ and $p$, to $r_1$ and $r_2$ as in the previous example. Then we apply variable speed 
reversal with pivoting \eqref{eq-pivotgnu} on each of the
 matrices for these variables.
 
 Here finding a suitable pivot map becomes necessary.
 We track the change of the profile by computing
 the relative $\ell^2$-norm difference between the previous and the current column,
 and pivots to the current column when it exceeds some threshold $\gamma$.
 That is, the pivot map is given by
 \[
\ell(j) = \max_k \left\{k \in \ZZ:  0 \le k \le j, \frac{\Norm{\ba_{j+1} - \ba_j}{2}}{\Norm{\ba_j}{2}} \ge \gamma \right\}
\]
For this example, setting $\gamma = 0.15$ was appropriate.

The achieved decay in singular values, along with the $L^2$-errors for 
each snapshot are shown in Figure \ref{fig:acoustic_htgs_singvals}.
Note how the error for the reversal is concentrated near the interface. 
Away from the interface,
the traveling wave solution is much more accurately captured with the reversal.
The decay of the singular values
can also be interpreted in this context.
While the decay is clearly more rapid than the POD modes, 
the difference is not as striking when compared to the case 
of homogeneous media.
The singular modes whose corresponding singular values belong to this 
trailing part represent the rapidly changing shape of the wave profile 
near the interface.
The slower decay is attributable to the presence of these modes.

A few sample reversal reconstruction are plotted along with the snapshot 
itself and the POD reconstruction in Figure \ref{fig:acoustic_htgs_recon}.
 $7$ and $5$ reduced basis 
vectors, for $r_1$ and $r_2$ respectively, were used for the reconstruction.
The accuracy of the reconstruction visibly deteriorates near the interface.

\begin{figure}
\centering
\begin{tabular}{cc} 
\includegraphics[width=0.45\textwidth]{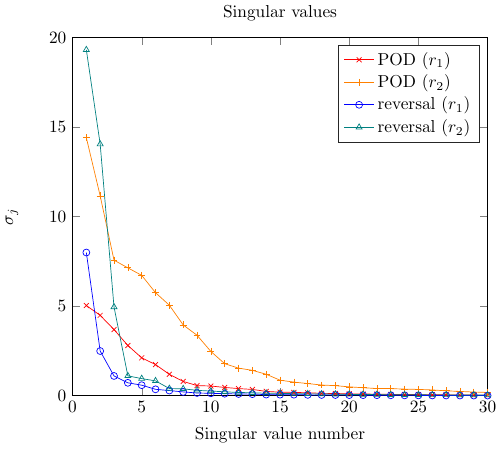}& 
\includegraphics[width=0.45\textwidth]{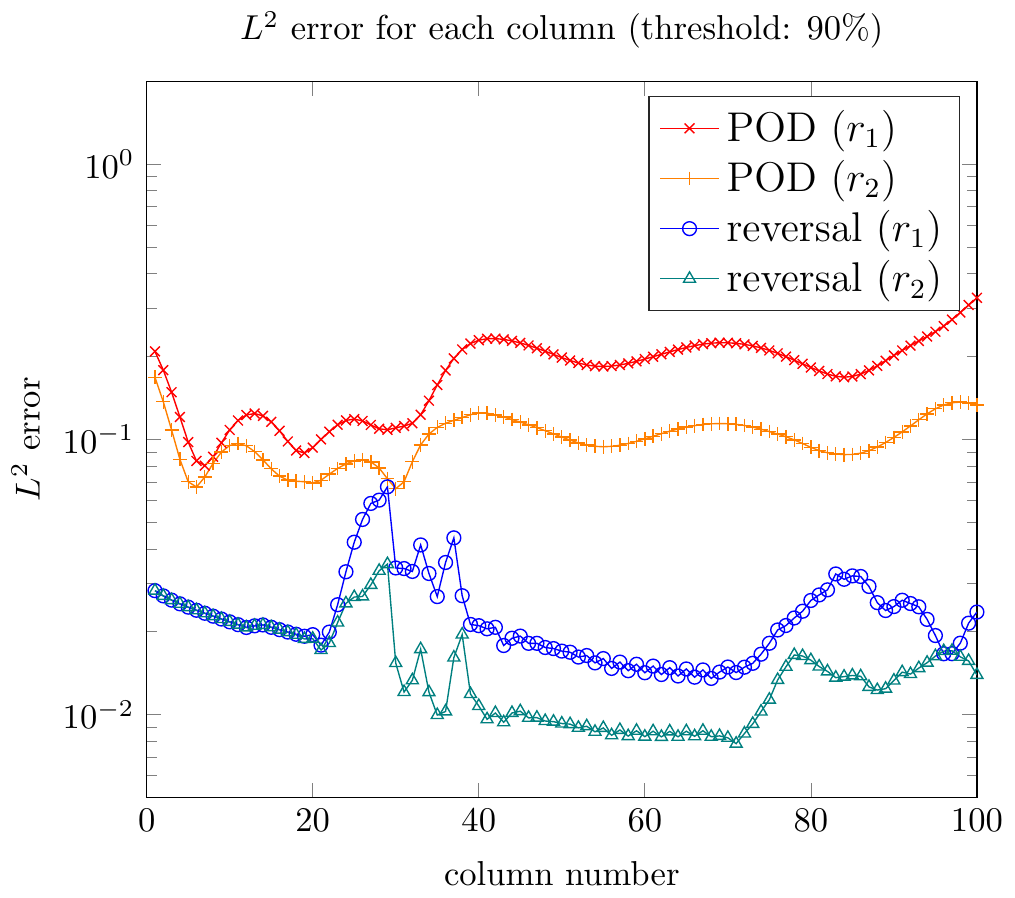} 
\end{tabular}
\caption{Faster decay of singular values of $r_1$ and $r_2$ for 
    the acoustic equation \eqref{eq-advection} with heterogeneous media (left).
    Largest $30$ singular values are shown. $5$ and $7$ singular values 
    represent $90\%$ threshold for these variables, respectively.
    $L^2$-error for each column of reconstruction for characteristic 
    variables $r_1$ and $r_2$.
    The error is concentrated near the snapshots in which the wave profile 
    is undergoing quick change near the interface.
}\label{fig:acoustic_htgs_singvals}
\end{figure}

\begin{figure}
    \centering
   \begin{tabular}{cc}
   \includegraphics[width=0.4\textwidth]{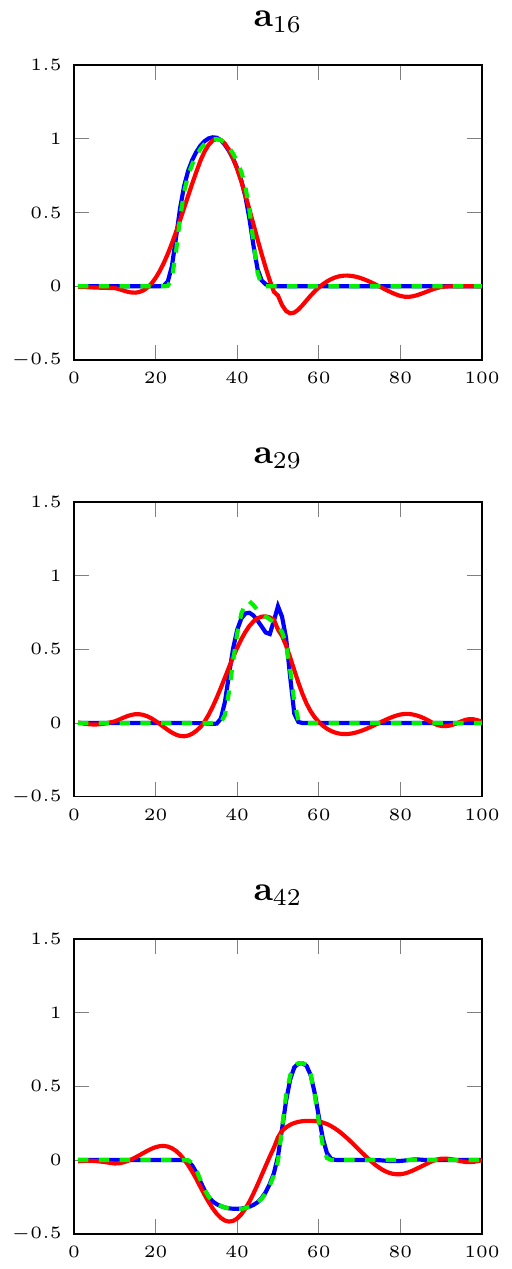}&
   \includegraphics[width=0.4\textwidth]{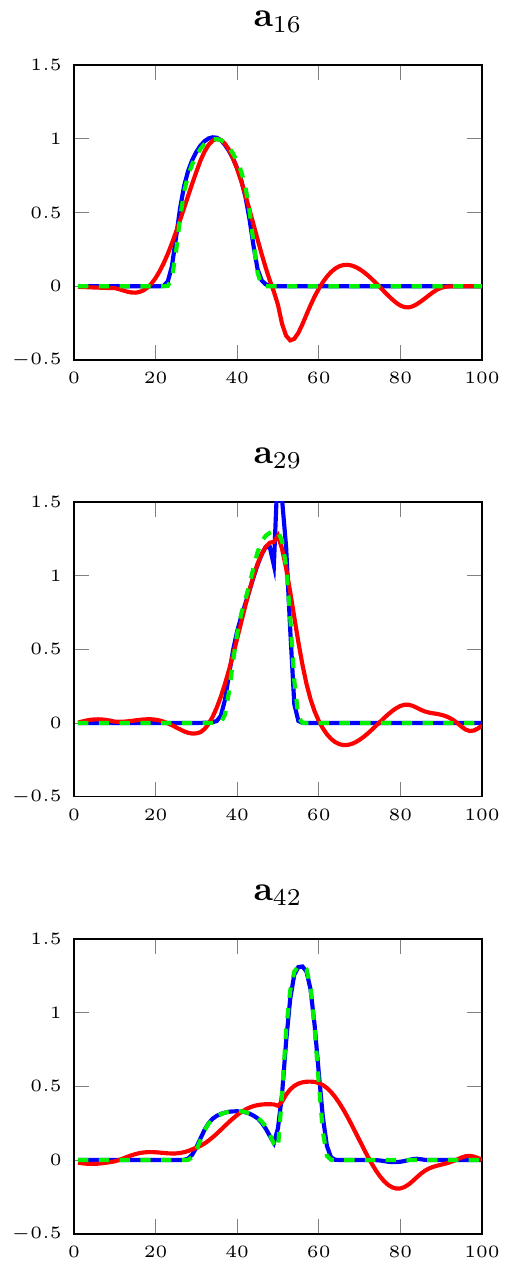}
   \end{tabular}
    \caption{Reconstruction of the solution to acoustic equation
    for state variable $p$ (left column) and $u$ (right column).
    The snapshot is given in dashed green, reversal reconstruction in blue, and
    POD reconstruction in red.
    Four reduced vectors were used for each reconstruction.}
    \label{fig:acoustic_htgs_recon}
\end{figure}

\section{Geometric interpretation}\label{sec:geo}

In this section, we present some geometric interpretations of 
the reversal procedure introduced in the Sections 
\ref{sec:rev} and \ref{sec:varlin} that arise naturally. 
Recall the matrices $\bfK(\nu)$ (Def. \ref{defi-upwindmat}), 
$\cK(\tilde{\nu}) = \cK(s,\nu)$ \eqref{eq-supwindmat} and
variable speed operator $\cK_c(\tilde{\nu})$ \eqref{eq-varmin}. Let
us define,
\beq
\cI_s(\ba) \equiv \prnc{\cK(s,\nu)\ba: \nu \in [0,1]} \text{ and }
\cM(\ba) \equiv \bigcup_{s \in \ZZ} \cI_s 
= \prnc{ \cK(\tilde{\nu})\ba : \tilde{\nu} \in \RR}
.
\label{eq-polys}
\eeq
Observe that $\cI_s(\ba)$
is the convex hull of $\{\bfK^s\ba, \bfK^{s+1} \ba\}$. Therefore,
given any column vector $\ba \in \RR^N$, $\cM(\ba)$
is a union of one-dimensional intervals lying in $\RR^N$ \eqref{eq-polys},
although $\cM(\ba)$ is not convex in $\RR^N$ in general.

In the minimization problem \eqref{eq-gnumin} we are choosing a point on 
this polygon that is closest to $\bb$.
Since $\bfK$ is an isometric map, the vertices of the polygon lie on the
$N$-sphere of radius $\Norm{\ba_j}{2}$.
Note that when choosing a point in the interior of $\cI_s(\ba)$, 
we are not preserving
$\Norm{\ba}{2}$ but $\sum_{j=1}^N (\ba)_j$ due to mass conservation 
in Lemma \ref{lem:upwind_basic} (d). 
This is also reflected in the $\cO(1/N^2)$ numerical diffusion term in 
$\cK(\tilde{\nu})^T \cK(\tilde{\nu}) - \bfI$ \eqref{eq-identity_error}.

When computing the reversal of a matrix $\bfA$, each column $\ba_j$ 
is transformed along its corresponding polygon $\cM(\ba_j)$.
The orientation of this polygon is determined by the column $\ba_j$ itself.  
On the other hand, it is easy to see that the $\cM(\ba)$ is always regular,
since the angle between the tangent vectors along its edges are
$-\ba^T \prn{\bfK + \bfK^T - 2 \bfI} \ba/\Norm{\bfD \ba}{2}^2$.

As a special case, if $\bfA = \bfI$,
all the regular polygons corresponding to each column vector of $\bfA$
coincide, and therefore the reversal is able to 
eliminate the functional \eqref{eq-gnumin}.
Then using a single reduced basis vector of $\mathring{\bfA}$
suffices.
This indicates that the decay of singular values of 
$\mathring{\bfA}$ depends on how well the polygons $\cM(\ba_j)$
are aligned with respected to each other.

It follows easily that the  $\cM(\ba)$ shrinks to a point
as $\ba$ approaches $\bone$. So the problem of finding the optimal
point in on the $\cM(\ba)$ becomes more constrained 
then finally becomes trivially ill-posed when $\ba$ or $\bb$ is 
parallel to $\bone$.
It is also easy to see that the shift number \eqref{eq-matrixgnu}
can always be found. That is, suppose $\ba,\bb \in \RR^N$,
then there always exists a shift number $\tilde{\nu}$ minimizing
$\Norm{\bb - \cK(\tilde{\nu})^T \ba}{2}$.
$\Norm{\bb - \bc}{2}^2$ for $\bc \in \cM(\ba)$ 
is a paraboloid on $\RR^N$ 
restricted to a compact subset, so it yields a minimum in $\cM(\ba)$. 
This minimum may not be unique, but the addition of a proper regularization 
term as in \eqref{eq-bnumin-penalty} will yield uniqueness for the problem.

Now, let us turn our attention to the dimension of the space
spanned by vertices of $\cM(\ba)$. 
The dimension depends on the periodicity of $\ba$, 
in particular when the period of $\ba$ is strictly smaller than $N$.
In the presence of such smaller periods, also called 
isotropy \cite{marsden1}, the minimization \eqref{eq-gnumin} can be further 
reduced;
and the smaller the period of $\ba$, the smaller the dimension should be.
This is eventually related to the period of the functional \eqref{eq-gnumin}, 
summarized in the next proposition and remarks that follow.

Note that a period of a function $g$ defined on $\RR$ is the smallest
number $L > 0$ such that $g(x + L) = g(x)$ for all $x \in \RR$.

\begin{prop}\label{prop:period}
	Suppose we are given vectors $\ba,\bb \in \RR^N$ both not parallel to $\bone$.
	Defining 
	$g(\tilde{\nu}) \equiv \bb - \cK(\tilde{\nu})^T \ba$,
	let us denote by $L$ the period of $g$, 
	$\cF$ the discrete Fourier transform,
	and $\text{gcd}$ the greatest common divisor.
	Then 
	\beq
	L = \frac{N}{\mathrm{gcd} \left[ \supp  \cF(\ba) \setminus \{0\} \right]}
	\label{eq-period}
	\eeq
\end{prop}

\begin{proof}
The proof is given in Appendix \ref{ap:proof}.
\end{proof}

Given this dimension $L$ in \eqref{eq-period}, we can  reduce the discrete set $W$ in \eqref{eq-recursion} by considering only the first $2L-1$ values.
The proposition characterizes the isotropy 
in a discrete case, whose continuous version was mentioned but not detailed in 
\cite{marsden1}.

In the variable speed case, the polygon  $\cM$ no longer retains its regularity.
Let us define a variable speed counter part to \eqref{eq-polys},
\beq
\cI_{c,s}(\ba) \equiv \prnc{\cK_c(\tilde{\nu})\ba: \nu \in [s,s+1]} \text{ and }
\cM_c(\ba) \equiv \bigcup_{s \in \ZZ} \cI_s = \prnc{ \cK_c(\tilde{\nu})\ba : \nu \in \RR}
.
\label{eq-varpolys}
\eeq
Mass is not preserved and $\cI_{c,s}$ is not guaranteed to be convex. 
The variable speed introduces
more vertices to the polygon $\cM_c(\ba)$ and $\cI_{c,s}(\ba)$ is itself now 
a union of more convex hulls. This makes the minimization
problem \eqref{eq-varmin} more challenging and causes the 
sharpening procedure \eqref{eq-recon} to run into difficulties, 
outside simple special cases for which one may impose
more boundary conditions \eqref{eq-dischelm} near the interface.

\section{Conclusion and future work}
This paper introduced a greedy algorithm that extracts the transport structure 
from the snapshot matrix by building on the template fitting strategy. 
Extensions of the algorithm though the generalizations of the shift operators 
were also considered. Numerical experiments show that the algorithm can capture 
complex hyperbolic behaviors in examples where shocks and interfaces 
are present.

The objective of this approach is to construct a reduced order model
of fully nonlinear hyperbolic problems, for use in high-dimensional 
applications arising in UQ and control design. In future work, the problem of
post-processing the output from the transport algorithm for use with 
existing projection-based model reduction methods will be investigated.
Extension of the algorithm to the multidimensional setting is currently
under development.

\section*{Acknowledgement}
We would like to thank Steven L. Brunton, Anne Greenbaum, and J. Nathan Kutz
for many helpful discussions.

\appendix
\section*{Appendices}
\addcontentsline{toc}{section}{Appendices}
\renewcommand{\thesubsection}{\Alph{subsection}}

\subsection{Proofs to Lemmas} \label{ap:proof}

Proof to Lemmas \ref{lem:upwind_basic} and \ref{lem:recursion},
Proposition \ref{prop:period} are given below.

\begin{proof}[Lemma \ref{lem:upwind_basic}]
\begin{itemize}
\item[(a)] Follows from the definition since $\bfI, \bfK$ commute,
\bals
\bfK(\nu) \bfK(\omega) 
&= (1-\nu)(1-\omega) \bfI + \nu(1-\omega) \bfK + \omega(1-\nu) \bfK  + \nu\omega\bfK^2\\
&=\prn{(1- \omega)\bfI + \omega \bfK}\prn{ (1-\nu) \bfI + \nu \bfK} 
= \bfK(\omega) \bfK(\nu).
\eals
\item[(b)] Again from the definition,
\bals
\bfK(\nu)^T \bfK(\omega) &= \prn{ (1-\nu) \bfI + \nu \bfK^T} 
\prn{(1- \omega)\bfI + \omega \bfK} \\
&=\prn{(1- \omega)\bfI + \omega \bfK}\prn{ (1-\nu) \bfI + \nu \bfK^T} 
= \bfK(\omega) \bfK(\nu)^T.
\eals
\item[(c)] We have
\[
\bfK(\nu) \bfK(\omega) = (1- \nu - \omega) \bfI + (\nu + \omega) \bfK
+ \frac{\nu \omega}{N^2}\frac{ \bfI - 2\bfK + \bfK^2}{1/N^2}.
\]
The matrix in the last term,
$\bfI - 2 \bfK + \bfK^2
= \bfK \prn{ \bfK^T + \bfK - 2\bfI} = \frac{1}{N^2}\bfK \bfL_h, $
is a shifted $\bfL_h$.
\item[(d)] This follows from the fact that the sum of the rows of $\bfK(\nu)$
is equal to $\begin{bmatrix} 1 & \cdots & 1 \end{bmatrix}$.
\end{itemize}
\end{proof}

\begin{proof}[Lemma \ref{lem:recursion}]
 Consider the $(+)$ case, the $(-)$ case follows similarly.
 \[
 \nu_s^+ =\argmin_{\omega \in \RR} \Norm{\bb - \prn{\bfK(\omega)\bfK^s}^T \ba}{2}^2
 = \argmin_{\omega \in \RR} \Norm{\bb - \bfK(\omega)^T \prn{\bfK^T}^s\ba}{2}^2.
 \]
 Using this in formula \eqref{eq-nu},
 \[
\nu_s^+ = \frac{1}{2} \prn{ 1 - 2 \frac{ \bb^T \prn{\bfK - \bfI} \prn{\bfK^T}^s\ba}{\ba^T \prn{\bfK + \bfK^T -2 \bfI} \ba}}.
 \]
Then we have
\[
\nu_{s+1}^+ = \frac{1}{2} \prn{ 1 - 2 \frac{ \bb^T \bfD \prn{\bfK^T}^{s+1}\ba}{\ba^T \bfL \ba}} 
= \nu_s^+ 
+ \frac{ \bb^T \bfL \prn{\bfK^T}^s\ba}{\ba^T \bfL \ba} = \nu_s^+ + \frac{ \prn{\bfK^s\bb}^T \bfL \ba}{\ba^T \bfL \ba}.
\]
\end{proof}

\begin{proof}[Proposition \ref{prop:period}]
	We will call $L$ a period of the vector $\ba$ if
	$L$ is the smallest  number such that $a_{j+L} = a_j$,
	where the indices are computed modulo $N$.

	It is easy to see that if $\ba$ has period $L$ then 
	$g(\tilde{\nu}+L) = g(\tilde{\nu})$. 
	
	Suppose $g$ has period $L$, then
	$L$ must be an integer. If $L$ has
	nonzero fractional part denoted by $\alpha$ then 
	\[
	g( L) = \bb - \prn{\bfK^T}^{L-\alpha} \bfK(\alpha)^T \ba 
		= \bb - \ba = g(0).
	\]
	So we must have that 
	$\prn{\bfK^T}^{L-\alpha}\bfK(\alpha)^T \ba = \ba$. Taking the 2-norm on both sides,
	\bals
		\Norm{\prn{\bfK^T}^{L-\alpha}\bfK(\alpha)^T \ba}{2}^2 &= 
		\ba^T \bfK(\alpha)^T \bfK(\alpha) \ba
		= \ba^T \ba - \frac{\alpha(1-\alpha)}{N^2} \ba^T \bfL \ba
		> \Norm{\ba}{2}^2,
	\eals
	for non-constant $\ba$ since $-\bfL$ is positive semi-definite
	with nullspace equal to that of constant vectors.
	Hence $\alpha$ cannot be in $(0,1)$ it so must be zero.
	
	Now since $L$ is an integer modulo $N$ let us assume $L>0$ without
	loss of generality, then
	\[
	\bb - \prn{\bfK^T}^\ell \ba = \bb - \prn{\bfK^T}^{\ell +L} \ba
	\quad \text{ for some } \quad \ell< N,
	\]
	which implies $\bfK^L\ba = \ba$, so $\ba$ has period $L$.
	Also, since $L$ is smallest number satisfying this equality,
	$L$ divides $N$. 

	Therefore we only need to find the dividend, 
	for which we simply take the discrete
	Fourier transform and compute the greatest common divisor
	of the nonzero frequencies. This yields the equation \eqref{eq-period}.
\end{proof}

\bibliographystyle{siamplain}
\bibliography{references}

\end{document}

%% file: macros.tex
\newcommand{\cF}{\mathcal{F}}

\newcommand{\cJ}{\mathcal{J}}

\newcommand{\cO}{\mathcal{O}}
\newcommand{\cT}{\mathcal{T}}
\newcommand{\cM}{\mathcal{M}}
\newcommand{\cP}{\mathcal{P}}

\newcommand{\cS}{\mathcal{S}}

\newcommand{\bnu}{\mbox{\boldmath$\nu$}}
\newcommand{\brho}{\boldsymbol{\rho}}

\newcommand{\bfRho}{\mbox{\boldmath$\mathrm{P}$}}
\newcommand{\beq}{\begin{equation}}
\newcommand{\eeq}{\end{equation}}
\def\bals#1\eals{\begin{align*} #1 \end{align*}}
\def\bal#1\eal{\begin{align} #1 \end{align}}
\newcommand{\dsp}[1]{\displaystyle{#1}}
\newcommand{\floor}[1]{\left\lfloor #1 \right\rfloor}

\newcommand{\Iin}{\quad \text{ in } }

\newcommand{\where}{\quad \text{ where } }
\newcommand{\Ffor}{\quad \text{ for }}
\newcommand{\Aand}{\quad \text{ and } \quad }

\newcommand\Dom\Omega

\newcommand\RR{\mathbb{R}}

\newcommand\ZZ{\mathbb{Z}}

\newcommand\Lap\Delta

\newcommand\abs[1]{\left\lvert #1 \right\rvert}

\newcommand\dx{\,\mathrm{d}x}

\newcommand\dt{\,\mathrm{d}t}

\def\bpde#1\epde{\[\left\{\begin{aligned}#1\end{aligned}\right. \]}
\def\inbpde#1\inepde{\left\{\begin{aligned}#1\end{aligned}\right.}
\def\binpde#1\einpde{\left\{\begin{aligned}#1\end{aligned}\right.}

\newcommand\prn[1]{\left( {#1} \right)}
\newcommand\prnc[1]{\left\{ {#1} \right\}}

\newcommand\brkt[2]{\left\langle {#1},{#2} \right\rangle}

\newcommand\Norm[2]{\left\lVert { #1 } \right\rVert_{#2}}

\def\cC{\mathcal{C}}

\def\cK{\mathcal{K}}
\def\cI{\mathcal{I}}

\def\bu{\mathbf{u}}
\def\bv{\mathbf{v}}

\def\bfA{\mathbf{A}}
\def\bfB{\mathbf{B}}

\def\bfD{\mathbf{D}}

\def\bfH{\mathbf{H}}
\def\bfI{\mathbf{I}}

\def\bfK{\mathbf{K}}

\def\bfP{\mathbf{P}}
\def\bfQ{\mathbf{Q}}
\def\bfR{\mathbf{R}}

\def\bfP{\mathbf{P}}
\def\bfU{\mathbf{U}}
\def\bfV{\mathbf{V}}

\def\bfI{\mathbf{I}}
\def\bfL{\mathbf{L}}
\def\bfU{\mathbf{U}}
\def\bfV{\mathbf{V}}

\def\bfSigma{\mbox{\boldmath$\Sigma$}}

\def\bone{\mathbf{1}}
\def\b0{\mathbf{0}}

\def\ba{\mathbf{a}}
\def\bb{\mathbf{b}}
\def\bc{\mathbf{c}}

\def\be{\mathbf{e}}
\def\bh{\mathbf{h}}

\def\br{\mathbf{r}}

\def\bv{\mathbf{v}}

\def\eps{\varepsilon}

\def\bbmat{\begin{bmatrix}[r]}
\def\ebmat{\end{bmatrix}}

\newcommand{\barr}{\begin{array}}
\newcommand{\ea}{\end{array}}
\newcommand{\bea}{\begin{eqnarray}}
\newcommand{\eea}{\end{eqnarray}}
\newcommand{\bt}{\begin{table}}
\newcommand{\et}{\end{table}}

\DeclareMathOperator\argmin{argmin}
\DeclareMathOperator\supp{supp}

\DeclareMathOperator\sign{sign}

\theoremstyle{plain}
\theoremstyle{definition}
\newtheorem{genericthm}{GENERIC THEOREM ENVIRONMENT}[section]

\newtheorem{lem}[genericthm]{Lemma}

\newtheorem{prop}[genericthm]{Proposition}
\newtheorem{notat}[genericthm]{Notation}

\newtheorem{defi}[genericthm]{Definition}

\numberwithin{equation}{section}

\newcommand{\TheTitle}{Transport Reversal for Model Reduction of Hyperbolic Partial Differential Equations}